\documentclass[a4paper,12pt]{article}

\usepackage[top=3.0cm,bottom=3.0cm,left=2.25cm,right=2.25cm]{geometry}
\usepackage{amsmath,amsthm,mathrsfs,graphicx,amsfonts}
\usepackage{bm}
\usepackage{cases}
\usepackage[english]{babel}
\usepackage{latexsym}
\usepackage{txfonts}
\usepackage[numbers,sort&compress]{natbib}
\usepackage[natural]{xcolor}
\usepackage{rotating}
\usepackage{mathtools}
\usepackage{enumitem}

\makeatletter
\def\@maketitle{
  \newpage
  \null
  \vskip 2em
  \begin{center}
  \let\footnote\thanks
    {\LARGE \@title \par}
    \vskip 1.5em
    {\large
      \lineskip .5em
      \parbox{\textwidth}{\centering\@author}
      \par
    }
    \vskip 1em
    {\large \@date}
  \end{center}
  \par
  \vskip 1.5em}
\makeatother

\newtheorem{lem}{Lemma}[section]
\newtheorem{thm}[lem]{Theorem}

\newtheorem{prob}[lem]{Problem}
\newtheorem{claim}[lem]{\indent Claim}
\newtheorem{conj}[lem]{Conjecture}

\usepackage[colorlinks=false, pdfborder={0 0 0}]{hyperref}

\begin{document}
\title{Paths of even length with equal-degree endpoints}
\date{}

\author{
Kaizhe Chen\thanks{School of the Gifted Young, University of Science and Technology of China, Hefei, Anhui 230026, China. 
Email: ckz22000259@mail.ustc.edu.cn}
~~~
Zhen Liu\thanks{ Center for Discrete Mathematics, Fuzhou University, Fuzhou, Fujian 350003, China. Email: 1552580575@qq.com}
~~~
Qinghou Zeng\thanks{ Center for Discrete Mathematics, Fuzhou University, Fuzhou, Fujian 350003, China. 
Email: zengqh@fzu.edu.cn}
}
\maketitle

\begin{abstract}
Addressing a question posed by Erd\H{o}s and Hajnal, Chen and Ma proved that, for all $n \ge 600$, the complete bipartite graph $K_{n,n+1}$ is the unique graph on $2n+1$ vertices with at least $n^2+n$ edges that contains no two vertices of equal degree joined by a path of length three. In this paper, we extend this result and prove that for every fixed integer \(\ell\ge 2\) and sufficiently large \(n\), the unique \(2n\)-vertex graph with at least \((n^2+n)/2\) edges that contains no two vertices of equal degree joined by a path of length \(2\ell\) is the half graph \(H_n\).
This resolves the problem posed by Chen and Ma, as well as a related question of Attwa, Az\'ocar Carvajal, Boyadzhiyska, Pierron, and Taraz concerning paths of even length with equal-degree endpoints.
\end{abstract}
\section{Introduction}\label{Intro}
A classical fact in graph theory asserts that every finite graph on at least two vertices contains two vertices of equal degree. It is then natural to ask what additional properties such vertices must satisfy under certain edge-density conditions. In 1991, Erd\H{o}s and Hajnal \cite{Erd1991} posed the following problem (see also \cite{Erd816}).

\begin{prob}[Erd\H{o}s and Hajnal \cite{Erd1991}]\label{EH}
Does every \((2n+1)\)-vertex graph with \(n^2+n+1\) edges contain two vertices of the same degree that are joined by a path of length three?
\end{prob}

The complete bipartite graph \(K_{n,n+1}\) shows that the edge bound in Problem~\ref{EH} is sharp. 
Chen and Ma \cite{CM2025} proved that for \(n \geq 600\), the graph \(K_{n,n+1}\) is the unique \((2n+1)\)-vertex graph with at least \(n^2+n\) edges that contains no two vertices of the same degree joined by a path of length three, and the $n< 600$ case was subsequently confirmed by Liu and Zeng \cite{LZ2025}, thereby fully resolving the above Erd\H{o}s--Hajnal problem.

A natural way to generalize this problem is to consider paths of arbitrary length. 
For any positive integers \(\ell\) and \(n\), let \(p_\ell(n)\) denote the maximum number of edges in an \(n\)-vertex graph that contains no two vertices of equal degree connected by a path of length \(\ell\). Chen and Ma \cite{CM2025} proposed the following conjecture and open problem.

\begin{conj}[Chen and Ma \cite{CM2025}]\label{Conj-Odd}
For any odd integer \(\ell \ge 3\) and sufficiently large \(n\), it holds that
$$p_\ell(2n+1)=n^2+n.$$
\end{conj}

\begin{prob}[Chen and Ma \cite{CM2025}]\label{Prob-CM}
Determine the exact value of \(p_{2\ell}(2n)\) for all \(\ell\) and sufficiently large \(n\).
\end{prob}

Regarding Conjecture \ref{Conj-Odd}, Liu and Zeng \cite{LZ2026} proved that for all $n \ge 11$, the case $\ell = 5$ holds, and $K_{n,n+1}$ is the unique extremal graph.
The conjecture was subsequently fully resolved by Zhao, Wang, and Lu \cite{ZWL2026}, 
who showed that for any odd $\ell \geq 7$ and sufficiently large $n$, 
$K_{n,n+1}$ is the unique extremal graph.
Shortly thereafter, Attwa, Azócar Carvajal, Boyadzhiyska, Pierron, and Taraz \cite{Attwa2026} gave a short proof that, for all $\ell \geq 6$ and  sufficiently large $n$, $p_\ell(n) \leq \left(1/4 + o(1) \right) n^2$. They also posed the following question regarding the asymptotic behavior of \(p_{2\ell}(n)\).

\begin{prob}[Attwa, Azócar Carvajal, Boyadzhiyska, Pierron, and Taraz \cite{Attwa2026}]\label{Prob-Attwa}
Is it true that $$p_{2\ell}(n) = \left( \frac{1}{8} + o(1) \right) n^2$$ for all $\ell$?
\end{prob}

The {\it half graph} \(H_n\), introduced by Erd\H{o}s and Hajnal \cite{Erdos1984,ErdosH1985}, is a bipartite graph on \(2n\) vertices with parts \(\{u_1,\dots,u_n\}\) and \(\{v_1,\dots,v_n\}\), where \(u_i\) is adjacent to \(v_j\) if and only if \(i \ge j\). Drawing on ideas from model theory, the half graph has become a central object in extremal combinatorics, structural graph theory, and theoretical computer science (see e.g. \cite{AlonFZ2019,ChudnovskyKOS2016,DreierEMMPT2024,MalliarisS2014,NesetrilMPRRS2021}).


Note that for any $\ell\ge 1$, the half graph $H_n$ contains no two vertices of equal degree joined by a path of length $2\ell$; thus $p_{2\ell}(2n) \ge (n^2+n)/2$. For the base case of $\ell = 1$, Chen and Ma \cite{CM2025} proved that $p_2(2n) = (n^2+n)/2$, and showed that the half graph is not the unique extremal graph. For $\ell\ge 2$, the exact value of $p_{2\ell}(2n)$ remains open.
In this paper, we prove the following result, which implies that for $\ell\ge 2$, $p_{2\ell}(2n)=(n^2+n)/2$ for sufficiently large $n$, and hence answers Problem \ref{Prob-CM}.

\begin{thm}\label{LZ}
For every integer $\ell\ge 2$ and sufficiently large $n$, the half graph $H_n$ is the unique $2n$-vertex graph with at least $\frac{n^2+n}{2}$ edges that contains no two vertices of the same degree joined by a path of length $2\ell$.
\end{thm}

In particular, Theorem~\ref{LZ} gives a new characterization of the half graph.
Note that adding an isolate vertex to a $(2n-1)$-vertex graph yields a $2n$-vertex graph with the same edge set.
Thus, Theorem~\ref{LZ} implies that \(p_{2\ell}(2n-1) \le p_{2\ell}(2n)=(n^2+n)/2\).
This gives an affirmative answer to Problem~\ref{Prob-Attwa}.


\medskip

\noindent\textbf{Notation}. Throughout the paper, we use the following notation. Let $G=(V(G),E(G))$ be a graph. For any \(v\in V(G)\), we denote by \(N_{G}(v)\) the set of \emph{neighbors} of \(v\) in \(G\), by \(N_{G}[v]\) the set \(N_{G}(v)\cup\{v\}\), and by \(d_G(v)\) the \emph{degree} of \(v\) in \(G\). 
For any \(S\subseteq V(G)\), let \(G[S]\) be the induced subgraph of \(G\) on \(S\). Let \(e_G(S)\) be the number of edges in \(G[S]\). 
For two disjoint sets \(X,Y\subseteq V(G)\), we write \(G[X,Y]\) for the bipartite subgraph of \(G\) with parts \(X\) and \(Y\) and edge set consisting of all edges of \(G\) between \(X\) and \(Y\). 
For arbitrary subsets \(X,Y\subseteq V(G)\), let \(e_G(X,Y)\) denote the number of edges with one endpoint in \(X\) and the other in \(Y\). 
In particular, $e_G(X,X)=e_G(X)$.
We will omit the subscript \(G\) in all these notations when there is no danger of confusion.
For any two distinct vertices \(s\) and \(t\) in $G$, we define
\[
\quad B_{st} := N(s) \cap N(t),\quad A^{(st)}_s := N(s)\setminus(\{t\} \cup B_{st}),\quad A^{(st)}_t := N(t)\setminus(\{s\} \cup B_{st}),
\]
\[{\rm and }\quad
D_{st} := V(G)\setminus\bigl(N[s]\cup N[t]\bigr).
\]
Fix an integer $\ell\ge 2$.
For any graph \(H\), let \(\operatorname{core}(H)\) denote the unique maximal induced subgraph of $H$ with minimum degree at least \( 2\ell+4\). 
\footnote{Since the union of any two induced subgraphs of \(H\) with minimum degree at least \(2\ell+4\) also has minimum degree at least \(2\ell+4\), \(\operatorname{core}(H)\) is well-defined.}
Define
\[
A_{s,1}^{(st)} := V\bigl(\operatorname{core}(G[A_s^{(st)}])\bigr),
\]
\[
A_{s,2}^{(st)} := V\bigl(\operatorname{core}(G[A_s^{(st)}\setminus A_{s,1}^{(st)}, A_t^{(st)}])\bigr) \cap (A_s^{(st)}\setminus A_{s,1}^{(st)}),
\]
\[
A_{s,3}^{(st)} := V\bigl(\operatorname{core}(G[A_s^{(st)}\setminus (A_{s,1}^{(st)}\cup A_{s,2}^{(st)}), D_{st}])\bigr) \cap \bigl(A_s^{(st)}\setminus (A_{s,1}^{(st)}\cup A_{s,2}^{(st)})\bigr),
\]
and 
\[
A_{s,4}^{(st)} := A_s^{(st)} \setminus (A_{s,1}^{(st)}\cup A_{s,2}^{(st)}\cup A_{s,3}^{(st)}).
\]
The subsets \(A_{t,1}^{(st)}, A_{t,2}^{(st)}, A_{t,3}^{(st)}, A_{t,4}^{(st)}\) are defined analogously. Define
$$Y_{st} := A_{s,4}^{(st)} \cup A_{t,4}^{(st)}\qquad {\rm and}\qquad A_{st} := (A_s^{(st)} \cup A_t^{(st)}) \setminus Y_{st}.$$

\section{Preliminary lemmas}
In this section, we present several lemmas that are useful in the proof of our main result.
\begin{thm}[Erd\H{o}s and Gallai \cite{EG1959}]\label{path-turan}
Let $G$ be a graph that contains no paths of length $k$. Then $|E(G)| \le \frac{k-1}{2} \cdot |V(G)|.$
\end{thm}

\begin{thm}[Naor and Verstra\"ete \cite{NV2005}]\label{bipartite-C2k}
For positive integers $m,n,k$, let $\operatorname{ex}(m,n,C_{2k})$ denote the maximum number of edges in a bipartite graph with parts of sizes $m$ and $n$ that contains no cycle of length $2k$. Then for each integer $k \ge 2$,
\[
\operatorname{ex}(m,n,C_{2k}) \le
\begin{cases}
(2k-3)\left[(mn)^{\frac{k+1}{2k}} + m + n\right] & \text{if } k \text{ is odd}, \\
(2k-3)\left[m^{\frac{k+2}{2k}} n^{\frac{1}{2}} + m + n\right] & \text{if } k \text{ is even}.
\end{cases}
\]
\end{thm}

\begin{lem}\label{lem3.2}
Let $\ell\ge 2$ be an integer.
Let \(n\) be a sufficiently large integer, and let $\alpha$ be an integer with \(\frac{2}{5}n < \alpha < n-300\ell\sqrt n\). Let \(S\) be a multiset of \(2n\) integers each at most \(n\), satisfying the following conditions:
\begin{enumerate}[label=(\alph*)]
\item For each integer \(d\) with \( d > \alpha\), there is at most one element equal to \(d\) in \(S\);
\item For each \(j \in \{1,2,3,4\}\) and each integer \(d\) satisfying
$d\ge \frac{2(n-\alpha)}{j}$,
there are at most \(j+1\) elements equal to \(d\) in \(S\).
\end{enumerate}
Then the sum of all elements in $S$ is at most
$ n^2 - 299\ell n^{3/2}$.
\end{lem}
\begin{proof}

We divide the proof into three cases according to the size of $\alpha$.

\medskip
\noindent \textbf{Case 1.} \( \frac23 n\le \alpha<n-300\ell\sqrt{n}\).

In this case, we have \(\alpha\ge 2(n-\alpha)\). Now construct an auxiliary multiset \(S^*\) as follows.  For each integer \(d\), the multiplicity of \(d\) in \(S^*\) is
\[
\begin{cases}
1, & \text{if } \alpha<d \le n,\\[4pt]
2, & \text{if } 2(n-\alpha) \le d \le \alpha,\\[4pt]
3, & \text{if } n-\alpha \le d \le 2(n-\alpha)-1,\\[4pt]
4, & \text{if } \bigl\lceil\frac{2}{3}(n-\alpha)\bigr\rceil \le d \le n-\alpha-1,\\[4pt]
5, & \text{if } \left\lfloor\frac{8}{15}(n-\alpha)\right\rfloor \le d \le \bigl\lceil\frac{2}{3}(n-\alpha)\bigr\rceil-1,\\[4pt]
0, & \text{otherwise.}
\end{cases}
\]
We first verify that \(|S^*|\) is at least \(|S|=2n\).  
By the definition of \(S^*\), we find
\[
\begin{aligned}
|S^*| &= 4n  - 2\alpha + 2 + \left\lceil \frac{2}{3}(n-\alpha) \right\rceil - 5 \left\lfloor \frac{8}{15}(n-\alpha) \right\rfloor \\
&\ge 4n  - 2\alpha + 2 +  \frac{2}{3}(n-\alpha)  - 5 \times \frac{8}{15}(n-\alpha)= 2n+2 > 2n.
\end{aligned}
\]

By the conditions satisfied by \(S\) and the definition of \(S^*\), for every integer \(d \ge \left\lfloor\frac{8}{15}(n-\alpha)\right\rfloor\), the multiplicity of \(d\) in \(S^*\) is at least the multiplicity of \(d\) in \(S\).  Moreover, \(\left\lfloor\frac{8}{15}(n-\alpha)\right\rfloor\) is the smallest element that appears in \(S^*\).  Consequently, the sum of the elements in \(S\) is at most the sum of the elements in \(S^*\). Elementary calculation shows that
\[
\begin{aligned}
\sum_{d\in S^*} d &=  \sum_{d=\alpha+1}^{n} d \;+\; 2 \sum_{d=2(n-\alpha)}^{\alpha} d \;+\; 3 \sum_{d=n-\alpha}^{2(n-\alpha)-1} d 
 \;+\; 4 \sum_{d=\lceil 2(n-\alpha)/3\rceil}^{n-\alpha-1} d \;+\; 5 \sum_{d=\lfloor 8(n-\alpha)/15\rfloor}^{\lceil 2(n-\alpha)/3\rceil-1} d \\
 &= \frac{113}{45}n^2 - \frac{181}{45}n\alpha + \frac{113}{45}\alpha^2 + O(n).
\end{aligned}
\]
Considering the above expression as a quadratic polynomial in \(\alpha\) with axis of symmetry \(\alpha = \frac{181}{226}n \), the condition \(\frac{2}{3}n \le \alpha < n-300\ell\sqrt{n}\) implies that it attains its maximum at \(\alpha = n-300\ell\sqrt{n}\). Hence,
\[
\begin{aligned}
\sum_{d\in S} d\le \sum_{d\in S^*} d &\le  n^2 - 300\ell n^{3/2} + O(n),
\end{aligned}
\]
which is at most $n^2  - 299\ell n^{3/2}$ for sufficiently large $n$. 

\medskip
\noindent \textbf{Case 2.} \( \frac12 n\le \alpha<\frac23 n\).

In this case, we have $n-\alpha\le \alpha <2(n-\alpha)$.
Similarly to Case 1, we construct an auxiliary multiset $S^*$ as follows.
For each integer $d$, the multiplicity of $d$ in $S^*$ is
\[
\begin{cases}
1, & \text{if } \alpha<d \le n,\\[4pt]
3, & \text{if } n-\alpha \le d \le \alpha,\\[4pt]
4, & \text{if } \bigl\lceil\frac{2}{3}(n-\alpha)\bigr\rceil \le d \le n-\alpha-1,\\[4pt]
5, & \text{if } \left\lfloor\frac{8}{15}(n-\alpha)\right\rfloor \le d \le \bigl\lceil\frac{2}{3}(n-\alpha)\bigr\rceil-1,\\[4pt]
2n-3\alpha, & \text{if} \ d=\left\lfloor\frac{8}{15}(n-\alpha)\right\rfloor-1,\\[4pt]
0, & \text{otherwise.}
\end{cases}
\]  
By the definition of \(S^*\), we have
\[
\begin{aligned}
|S^*| &=4n - 2\alpha + 3 + \left\lceil \frac{2}{3}(n-\alpha) \right\rceil - 5\left\lfloor \frac{8}{15}(n-\alpha) \right\rfloor \\
&\ge 4n - 2 \alpha + 3 + \frac{2}{3}(n-\alpha)  - 5 \times\frac{8}{15}(n-\alpha)=2n+3> 2n.
\end{aligned}
\]

Similarly, since the multiplicity of \(d\) in \(S^*\) is at least the multiplicity of \(d\) in \(S\) for any \(d \ge \left\lfloor\frac{8}{15}(n-\alpha)\right\rfloor\), and \(\left\lfloor\frac{8}{15}(n-\alpha)\right\rfloor-1\) is the smallest element  in \(S^*\), we deduce that the sum of the elements in \(S\) is at most the sum of the elements in \(S^*\). Note that 
\[
\begin{aligned}
\sum_{d\in S^*} d &=  \sum_{d=\alpha+1}^{n} d  \;+\; 3 \sum_{d=n-\alpha}^{\alpha} d  \;+\; 4 \sum_{\lceil2(n-\alpha)/3\rceil}^{n-\alpha-1} d   
\;+\;  5 \sum_{d=\lfloor 8(n-\alpha)/15\rfloor}^{\lceil 2(n-\alpha)/3\rceil-1} d  + (2n-3\alpha)\left(\left\lfloor\frac{8}{15}(n-\alpha)\right\rfloor-1\right) \\ 
&= \frac{71}{45} n^2 - \frac{121}{45} n\alpha + \frac{47}{18} \alpha^2 + O(n).
\end{aligned}
\]
Since \(\frac{1}{2}n \le \alpha < \frac{2}{3}n\), 
the above quadratic polynomial in \(\alpha\) attains its maximum at \(\alpha =\frac {2}{3}n\). Hence,
\[
\begin{aligned}
\sum_{d\in S} d\le \sum_{d\in S^*} d &\le \frac{383}{405}n^2 + O(n),
\end{aligned}
\]
which is at most $n^2  - 299\ell n^{3/2}$ for sufficiently large $n$.

\medskip
\noindent \textbf{Case 3.} \( \frac{2}{5} n< \alpha< \frac{1}{2}n\).

In this case, we have $\bigl\lceil\frac{2}{3}(n-\alpha)\bigr\rceil\le\alpha <n-\alpha.$
Similarly, we define $S^*$ as follows.
For each integer $d$, the multiplicity of $d$ in $S^*$ is
\[
\begin{cases}
1, & \text{if } \alpha<d \le n,\\[4pt]

4, & \text{if } \bigl\lceil\frac{2}{3}(n-\alpha)\bigr\rceil \le d \le \alpha,\\[4pt]
5, & \text{if } \left\lfloor\frac{8}{15}(n-\alpha)\right\rfloor \le d \le \bigl\lceil\frac{2}{3}(n-\alpha)\bigr\rceil-1,\\[4pt]
3n-5\alpha, & \text{if} \ d=\left\lfloor\frac{8}{15}(n-\alpha)\right\rfloor-1,\\[4pt]
0, & \text{otherwise.}
\end{cases}
\]
Then,
\[
\begin{aligned}
|S^*| &=4n - 2\alpha + 4 + \left\lceil\frac{2(n-\alpha)}{3}\right\rceil - 5\left\lfloor\frac{8(n-\alpha)}{15}\right\rfloor \\
&\ge 4n - 2 \alpha + 4 + \frac{2}{3}(n-\alpha)  - 5 \times\frac{8}{15}(n-\alpha)=2n+4 >2n.
\end{aligned}
\]
Similarly, the sum of the elements in \(S\) is at most the sum of the elements in \(S^*\). We calculate that
\[
\begin{aligned}
\sum_{d\in S^*} d &=  \sum_{d=\alpha+1}^{n} d   \;+\; 4 \sum_{\lceil2(n-\alpha)/3\rceil}^{\alpha} d    
\;+\;  5 \sum_{d=\lfloor 8(n-\alpha)/15\rfloor}^{\lceil 2(n-\alpha)/3\rceil-1} d   + (3n-5\alpha) \left(\left\lfloor\frac{8}{15}(n-\alpha)\right\rfloor-1\right) \\ 
&= \frac{29}{18}n^2-\frac{148}{45}n\alpha+\frac{331}{90}\alpha^2 + O(n).
\end{aligned}
\]
By the Case assumption, the above expression attains its maximum at \(\alpha =\frac {1}{2}n\). So,
\[
\begin{aligned}
\sum_{d\in S} d\le \sum_{d\in S^*} d &\le \frac{319}{360}n^2 + O(n),
\end{aligned}
\]
which is at most $n^2  - 299\ell n^{3/2}$ for sufficiently large $n$.
This completes the proof.
\end{proof}

\section{Proof of main result}
Let $\ell\ge 2$ be a fixed integer, let $n$ be a sufficiently large integer, and let $G$ be a graph on $2n$ vertices with at least $(n^2+n)/2$ edges. Suppose that no two vertices of equal degree in $G$ are joined by a path of length $2\ell$. Our final aim is to prove that $G$ is isomorphic to the half graph $H_n$.

\subsection{Large degree vertices}\label{sub1}
Let $R$ be the set of all vertices of degree at least $n+\ell+1$ in $G$, and set $r:=|R|$.
Note that the number of common neighbors of any two distinct vertices in \(R\) is at least
\begin{equation}\label{eq:common-neighborhood-q}
2(n+\ell+1)-2n
=2\ell+2.
\end{equation}
\begin{lem}\label{2kL+}
    If $r \ge \ell + 1$, then all vertices in $R$ have pairwise distinct degrees.
\end{lem}
\begin{proof}
    Suppose, to the contrary, that there exist two distinct vertices $x,y\in R$ such that $d(x)=d(y)$. Since $r\ge \ell+1$, we may choose $\ell-1$ distinct vertices
$z_1,z_2,\ldots,z_{\ell-1}\in R\setminus\{x,y\}$.
Put $z_0=x$ and $z_\ell=y$. 
By \eqref{eq:common-neighborhood-q}, we can greedily select vertices $a_t\in N(z_{t-1})\cap N(z_t)$, for each $1\le t\le \ell$, such that the $2\ell +1$ vertices $z_0,a_1,z_1,a_2,\ldots,z_{\ell-1},a_\ell,z_\ell$ are pairwise distinct. 
Consequently, $z_0a_1z_1a_2\cdots z_{\ell-1}a_\ell z_\ell$ forms a path of length $2\ell$ whose endpoints $z_0=x$ and $z_\ell=y$ share identical degrees, yielding a contradiction.
\end{proof}

By Lemma \ref{2kL+}, if $r\ge \ell+1$, we may label $R=\{v_1,v_2,\ldots,v_r\}$ such that
\[
d(v_1)>d(v_2)>\cdots>d(v_r)\ge n+\ell+1.
\]
In particular, for every \(j\in\{1,2,\ldots,r\}\),
\begin{equation}\label{eq:degree-lower-q}
d(v_j)\ge n+\ell+1+r-j.
\end{equation}

\begin{lem}\label{differentdegree}
Assume that \(r\ge \ell+1\). For any two distinct indices \(i,j\in\{1,2,\ldots,r\}\), and any two distinct vertices $s\in N(v_i),\ t\in N(v_j)$, we have $d(s)\ne d(t)$.
\end{lem}

\begin{proof}
If both $s$ and $t$ belong to $R$, then Lemma \ref{2kL+} gives $d(s)\ne d(t)$.
If exactly one of \(s,t\) belongs to \(R\), then $d(s)\ne d(t)$,
by the definition of $R$. Hence, we may assume that $s,t\notin R.$

Since \(r\ge \ell+1\), we can choose \(\ell-2\) distinct vertices
$z_1,z_2,\ldots,z_{\ell-2}\in R\setminus\{v_i,v_j\}.$
For \(\ell=2\), this list is empty. Put $z_0=v_i,\ z_{\ell-1}=v_j.$
By \eqref{eq:common-neighborhood-q}, we can choose $a_q\in N(z_{q-1})\cap N(z_q)$, for each \(q\in\{1,2,\ldots,\ell-1\}\), so that the $2\ell +1$ vertices $t,z_0,a_1,z_1,a_2,\ldots,z_{\ell-2},a_{\ell-1},z_{\ell-1},s$ are distinct. 
Therefore $sz_0a_1z_1a_2\cdots a_{\ell-1}z_{\ell-1}t$ is a path of length $2\ell$ with endpoints \(s\) and \(t\). So, \(d(s)\ne d(t)\).
\end{proof}

\begin{lem}\label{eq:v1bound}
    There are at most $n+\ell+r$ vertices in $G$ with pairwise distinct positive degrees.
\end{lem}
\begin{proof}
    Since every vertex outside \(R\) has degree at most \(n+\ell\), we obtain the desired result.
\end{proof}

For each \(i\in\{1,2,\ldots,r\}\), set
\[
U_{v_i}:=N(v_i)\setminus
\bigcup_{\substack{1\le j\le r\\ j\ne i}}N[v_j].
\]
By definition, we have $U_{v_i}\cap R=\emptyset$ for every \(i\), and
$U_{v_i}\cap U_{v_j}=\emptyset$ for all $i\ne j$.

\begin{lem}\label{lem:U}
Assume that \(r\ge \ell+1\). For every \(i\in\{1,2,\ldots,r\}\), the set \(U_{v_i}\) contains two vertices of the same degree.
\end{lem}

\begin{proof}
Suppose not, and let \(k\in\{1,2,\ldots,r\}\) be the smallest index such
that all vertices in \(U_{v_k}\) have pairwise distinct degrees. Then, for each
\(i<k\), the set \(U_{v_i}\) contains two vertices of the same degree; in
particular, \(U_{v_i}\ne\emptyset\). Choose one vertex $u_i\in U_{v_i}$ for each \(i<k\), and define
\[
S:=N[v_k]\cup\{u_1,u_2,\ldots,u_{k-1}\}.
\]
Since \(u_i\in U_{v_i}\), we have \(u_i\notin N[v_k]\) for every \(i<k\).
Moreover, the sets \(U_{v_i}\) are pairwise disjoint. Thus,
\begin{equation}\label{eq:S-size-37}
|S|=|N[v_k]|+k-1=d(v_k)+k.
\end{equation}

We claim that all vertices in \(S\) have pairwise distinct degrees. Suppose,
for a contradiction, that there exist two vertices \(x,y\in S\) with
$d(x)=d(y)$. Similarly, by Lemma \ref{2kL+} and the definition of $R$, we may assume that $x,y\notin R$. Hence, neither $x$ nor $y$ is $v_k$.
First, suppose that both \(x\) and \(y\) are among the chosen vertices \(u_1,\ldots,u_{k-1}\), say \(x=u_a\) and \(y=u_b\) for some $a\ne b$. 
Then \(u_a\in N(v_a)\) and \(u_b\in N(v_b)\), and Lemma \ref{differentdegree} gives \(d(x)\ne d(y)\), a contradiction. Next, suppose that exactly one of \(x,y\) is among \(u_1,\ldots,u_{k-1}\). Without loss of generality, let \(x=u_a\) for some \(a<k\), and let \(y\in N(v_k)\). Again, applying Lemma \ref{differentdegree} with \(x=u_a\in N(v_a)\) and \(y\in N(v_k)\) gives \(d(x)\ne d(y)\), a contradiction.
Finally, suppose that \(x,y\in N(v_k)\). If both \(x\) and \(y\) belong to \(U_{v_k}\), then their degrees would be distinct by the choice of \(k\). Hence, without loss of generality, we may assume $x\in N(v_k)\setminus U_{v_k}.$ By the definition of \(U_{v_k}\), there exists some \(j\neq k\) such that \(x\in N[v_j]\).
Since \(x\notin R\), we have $x\in N(v_j).$ Together with \(y\in N(v_k)\), Lemma \ref{differentdegree} again yields a contradiction.
Therefore, all vertices in \(S\) have pairwise distinct degrees, as claimed. 

By the definition of $S$, every vertex in $S$ has positive degree.
Together with the above claim and Lemma~\ref{eq:v1bound}, we have $|S|\le n+\ell+r$.
Combining with \eqref{eq:degree-lower-q} and \eqref{eq:S-size-37}, we obtain
\[
n+\ell+r\ge |S|=d(v_k)+k\ge n+\ell+r+1,
\]
which is impossible. This proves the lemma.
\end{proof}

\begin{lem}\label{r2n}
We have $r <2\sqrt{n}$.
\end{lem}

\begin{proof}
If $r\le \ell$, then $r\le \ell<2\sqrt{n}$ for sufficiently large $n$.
Hence, we may assume that $r\ge \ell+1$, to which Lemmas \ref{differentdegree} and \ref{lem:U} apply.

By Lemma \ref{lem:U}, for each \(i\in\{1,2,\ldots,r\}\), we can choose one vertex $u_i\in U_{v_i}.$ Fix \(j\in\{1,2,\ldots,r\}\), and define
\[
S_j:=\bigl(N[v_j]\setminus U_{v_j}\bigr)\cup\{u_1,u_2,\ldots,u_r\}.
\]
The two parts in this union are disjoint. Indeed, \(u_j\in U_{v_j}\), while for
\(i\ne j\), the definition of \(U_{v_i}\) gives \(u_i\notin N[v_j]\). Thus,
\begin{equation}\label{eq:Sj-size-38}
|S_j|=|N[v_j]\setminus U_{v_j}|+r=d(v_j)+1-|U_{v_j}|+r.
\end{equation}
By the same argument used in the proof of Lemma \ref{lem:U} to show that the vertices in $S$ have pairwise distinct degrees, we may apply Lemma \ref{differentdegree} to conclude that the vertices in $S_j$ are of pairwise distinct degrees. 
Note that every vertex in $S_j$ has positive degree.
By Lemma~\ref{eq:v1bound}, we find $|S_j|\le n+\ell+r$.
It follows from \eqref{eq:Sj-size-38} that
\[n+\ell+r\ge |S_j|=d(v_j)+1-|U_{v_j}|+r. \]
Using \eqref{eq:degree-lower-q}, we further obtain
\begin{equation}\label{eq:Uvj-lower-q}
|U_{v_j}|\ge d(v_j)+1-n-\ell \ge r-j+2.
\end{equation}

Since the sets \(U_{v_1},U_{v_2},\ldots,U_{v_r}\) are pairwise disjoint, we deduce by \eqref{eq:Uvj-lower-q} that
\[
2n = |V(G)| \ge \sum_{j=1}^{r} |U_{v_j}|\ge \sum_{j=1}^{r} (r-j+2)
= \frac{r(r+3)}{2}.
\]
This gives \(r < 2\sqrt{n}\), completing the proof.
\end{proof}

\begin{lem}\label{lem:R-minus-small}
The number of vertices with degree greater than $n$ is less than $3\sqrt{n}$.
\end{lem}

\begin{proof}
Set $R^*:=\{v\in V(G): n<d(v)<n+\ell+1\}$.
By Lemma \ref{r2n}, it suffices to prove \(|R^*|< \sqrt n\).
Suppose that \(|R^*|\ge \sqrt n\). 
By the Pigeonhole Principle, there exist \(q\in\{n+1, \ldots ,n+\ell \}\) and \(A\subseteq R^*\) such that
\[
d(x)=q\ \ \text{for all }x\in A,
\qquad
|A|=\left\lceil \frac{\sqrt n}{\ell} \right\rceil.
\]
Consider the subgraph $J:=G[A,V(G)\setminus A].$ 
For every \(a\in A\), $d_J(a)\ge d_G(a)-(|A|-1)\ge n-|A|$. Thus, $$|E(J)|\ge |A|(n-|A|)>4n\ell =2\ell|V(J)|$$ for sufficiently large \(n\). 
By Theorem~\ref{path-turan}, $J$ contains a path $P$ of length $2\ell+1$. Since \(J\) is bipartite, $P$ contains a path $P'$ of length $2\ell$ with both endpoints lying in \(A\). 
However, both endpoints have degree \(q\), which is a contradiction.
\end{proof}

\subsection{Equal degree vertices}
Throughout this section, for two fixed vertices $s$ and $t$, we simply write $A_{s}:=A_{s}^{(st)}$, $A_{t}:=A_{t}^{(st)}$, $A_{s,i}:=A_{s,i}^{(st)}$, and $A_{t,i}:=A_{t,i}^{(st)}$, for each $i\in \{1,2,3,4 \}$.

\begin{lem}\label{lem}
Let $s$ and $t$ be two vertices of equal degree in $G$. Then the following statements hold.
\end{lem}
\begin{enumerate}[label=(\alph*)] 
\item Let \(x\in A_s\setminus A_{s,4}\). For every set \(F\subseteq V(G)\setminus\{x\}\) with \(|F|\le 6\) and every non-negative even integer \(k\le 2\ell-2\), there exists a path of length \(k\) starting at \(x\), ending in \(A_s\), and avoiding \(F\). The analogous statement holds with \(s\) and \(t\) interchanged.
\item  We have \(\sum_{v\in B_{st}} d_G(v) \le (6\ell-2)n\). Furthermore, when \(\ell=2\), we have \(|B_{st}| \le 5\sqrt{n}\).
\item   We have $\sum_{v\in Y_{st}} d_G(v) \le 60\ell n$.
\item  Vertices in $A_s\setminus A_{s,4}$ have pairwise distinct degrees; the same holds for $A_t\setminus A_{t,4}$.
\item  If $B_{st}\neq\emptyset$, then all vertices in $A_{st}$ have pairwise distinct degrees.
\end{enumerate}
\begin{proof}
(a) The case \(k=0\) is clear. Let \(k\ge 2\). Since \(x\in A_s\setminus A_{s,4}\), \(x\) belongs to one of \(A_{s,1},A_{s,2},A_{s,3}\). 
Let \(H\) be $\operatorname{core}(G[A_s])$ if $x\in A_{s,1}$,  be $\operatorname{core}(G[A_s\setminus A_{s,1},A_t])$ if $x\in A_{s,2}$, and be $\operatorname{core}(G[A_s\setminus(A_{s,1}\cup A_{s,2}),D_{st}])$ if $x\in A_{s,3}$. 
Then \(x\in V(H)\) and \(H\) has minimum degree at least \( 2\ell+4\). We greedily build a path \(v_0v_1\cdots v_k\) in \(H\), with \(v_0=x\), avoiding \(F\). 
Assume that \(v_0,\ldots,v_i\) have been chosen for some \(i<k\).
Since $|F\cup\{v_0,\ldots,v_{i-1}\}|\le 6+(k-1)\le 2\ell+3$ and $d_H(v_i)\ge 2\ell +4$, we can find a neighbor of $v_i$ avoiding $F\cup\{v_0,\ldots,v_{i-1}\}$.
Thus we obtain a path of length \(k\) starting at \(x\) and avoiding \(F\). If \(x\in A_{s,1}\), this path lies in \(A_s\). If \(x\in A_{s,2}\) or \(x\in A_{s,3}\), since $H$ is bipartite and  \( k\) is even, the endpoint other than $x$ also lies in \(A_s\). The statement with \(s,t\) interchanged is identical.

(b) We claim that $G[B_{st}]$ contains no path of length $2\ell-2$. Indeed, a path $x_0x_1\cdots x_{2\ell-2}$ inside $B_{st}$ would give rise to $s x_0x_1\cdots x_{2\ell-2} t$, a path of length $2\ell$ whose endpoints $s,t$ have equal degree. Thus, by Theorem \ref{path-turan},
\begin{equation}\label{eq:B-internal-small}
e(B_{st}) \le (\ell-1)|B_{st}|\le (2\ell-2)n .
\end{equation}

Next consider the bipartite graph $H_{B_{st}} := G\bigl[B_{st},\ V(G)\setminus\bigl(B_{st}\cup\{s,t\}\bigr)\bigr]$. 
We claim that $H_{B_{st}}$ contains no paths of length $2\ell -1$. Otherwise, assume that $H_{B_{st}}$ contains a path $P$ of length $2\ell -1$. Since $H_{B_{st}}$ is bipartite, $P$ contains a path $P'$ of length $2\ell -2$ with both endpoints lying in \(B_{st}\). 
Then $sP't$ is a path of length $2\ell$ with equal-degree endpoints, a contradiction.
Hence, Lemma \ref{path-turan} gives
\begin{equation}\label{eq:B-external-small}
e(H_{B_{st}}) \le(\ell-1)|H_{B_{st}}|\le (2\ell-2)n.
\end{equation}
Together with the edges between $B_{st}$ and $\{s,t\}$, equations \eqref{eq:B-internal-small} and \eqref{eq:B-external-small} yield 
\begin{equation}\label{Bst}
    \sum_{v\in B_{st}} d_G(v)=2|B_{st}|+2e(B_{st})+e(H_{B_{st}}) \le(6\ell-2)n.
\end{equation}

We now consider the case $\ell=2$. If $|B_{st}| \le 2$, then $|B_{st}|\le 5\sqrt{n}$ holds. Assume that $|B_{st}| \ge 3$. We claim that the degrees of all vertices in $B_{st}$ are pairwise distinct. Indeed, if two vertices $u,v\in B_{st}$ have the same degree, then picking a third vertex $w\in B_{st}\setminus\{u,v\}$ would yield a path $uswtv$ of length $4$ with equal-degree endpoints, a contradiction. Note that every vertex in $B_{st}$ is adjacent to both $s$ and $t$, and hence has degree at least $2$. Consequently, 
\[
\sum_{v\in B_{st}} d_G(v) \ge \sum_{i=2}^{|B_{st}|+1} i = \frac{|B_{st}|^2 + 3|B_{st}|}{2}.
\] 
Combining the above inequality with \eqref{Bst}, we obtain that \(|B_{st}| \le 5\sqrt{n}\).

(c) We claim that, for any graph \(H\), $e(V(H), V(H)\setminus V(\operatorname{core}(H)))\le (2\ell+3)|V(H)|$.
Indeed, if we iteratively delete vertices in $H$ of current degree less than \(2\ell+4\) until no such vertex remains, then the remaining subgraph is \(\operatorname{core}(H)\), as each vertex in \(\operatorname{core}(H)\) won't be deleted during the process. Moreover, $e(V(H), V(H)\setminus V(\operatorname{core}(H)))$ is the number of edges we deleted during the process. Since each deleted vertex removes at most \(2\ell+3\) edges, we find $e(V(H), V(H)\setminus V(\operatorname{core}(H)))\le (2\ell+3)|V(H)|$, as claimed.

Note that \(2n\ge |A_s|+|A_t|= 2|A_s|\). Applying the above claim with \(H = G[A_s]\), we have
\[
e_G(A_s,A_{s,4}) \le e(V(H), V(H)\setminus V(\operatorname{core}(H))) \le (2\ell+3)|A_s| \le (2\ell+3)n.
\]
Similarly, taking \(H = G[A_s \setminus A_{s,1}, A_t]\) yields
\[
e_G( A_t,A_{s,4}) \le e(V(H), V(H)\setminus V(\operatorname{core}(H))) \le (4\ell+6)n.
\]
Moreover, taking \(H=G[A_s \setminus (A_{s,1} \cup A_{s,2}), D_{st}]\) leads to
\[
e_G(D_{st},A_{s,4}) \le e(V(H), V(H)\setminus V(\operatorname{core}(H))) \le (4\ell+6)n.
\]
Note that in the above three cases, we have used the fact that $A_{s,4}\subseteq V(H)\setminus V(\operatorname{core}(H)))$ by the definition of $A_{s,4}$.
According to Lemma~\ref{lem}(b), we have
\[
e_G(B_{st},A_{s,4}) \le \sum_{v \in B_{st}} d_G(v) \le (6\ell-2)n.
\]
Together with the $|A_{s,4}|$ edges between \(A_{s,4}\) and \(\{s,t\}\), we deduce that
\begin{align*}
\sum_{v \in A_{s,4}} d_G(v) \le 2e_G(A_s,A_{s,4}) + e_G(A_t,A_{s,4}) + e_G(D_{st},A_{s,4}) + e_G(B_{st},A_{s,4}) + |A_{s,4}| \le 30\ell n.
\end{align*}
By symmetry, $\sum_{v\in A_{t,4}}d_G(v)\le 30\ell n$. The desired estimate then follows as $Y_{st}=A_{s,4}\cup A_{t,4}$.

(d) Let \(x,y\) be any two vertices in \(A_s\setminus A_{s,4}\). Applying (a) with \(F=\{s,y\}\) and \(k=2\ell-2\), we find a path $x_0x_1\cdots x_{2\ell-2}$ avoiding \(F\), with $x_0=x$ and \(x_{2\ell-2}\in A_s\). Then $x_0x_1\cdots x_{2\ell-2}sy$ is a \(2\ell\)-path, which implies $d(x)\ne d(y)$. Thus vertices in \(A_s\setminus A_{s,4}\) have pairwise distinct degrees. The proof for \(A_t\setminus A_{t,4}\) is identical.

(e) Assume \(B_{st}\neq\emptyset\), and choose \(w\in B_{st}\). By (d), it remains to exclude equal-degree pairs $x\in A_s\setminus A_{s,4}$ and $y\in A_t\setminus A_{t,4}$. Using (a) with \(F=\{y,s,t,w\}\) and \(k=2\ell-4\), we obtain a path $x_0x_1\cdots x_{2\ell-4}$ avoiding \(F\), with $x_0=x$ and \( x_{2\ell-4}\in A_s\). Hence, $x_0x_1\cdots x_{2\ell-4}s w t y$ is a \(2\ell\)-path, yielding $d(x)\ne d(y)$. This completes the proof.
\end{proof}

\begin{lem}\label{eq:distinct-degree-sum}
    Let \(F\subseteq V(G)\) be a set of vertices of pairwise distinct degrees. Then,
    \begin{equation*}
\sum_{v\in F}d_G(v)\le n|F|-\frac{|F|^2}{2}+O(n^{3/2}) \notag\le \frac{n^2}{2}+O(n^{3/2}).
\end{equation*}
\end{lem}
\begin{proof}
Let \(q\) be the number of vertices in \(F\) having degree greater than \(n\).
By Lemma~\ref{lem:R-minus-small}, we have \(q< 3\sqrt n\). The remaining vertices have distinct degrees in \(\{0,1,\ldots,n\}\). Thus, 
\begin{equation*}
\sum_{v\in F}d_G(v)\le 2nq+\sum_{i=0}^{|F|-q-1}(n-i) =n|F|-\frac{|F|^2}{2}+O(n^{3/2}) \le \frac{n^2}{2}+O(n^{3/2}),
\end{equation*}
where the last inequality follows by the Cauchy-Schwarz inequality.
\end{proof}

\begin{lem}\label{B=kong}
Let $s$ and $t$ be two vertices of equal degree at least \(n-n^{4/5}\). Then $B_{st}=\emptyset.$
\end{lem}

\begin{proof}
Suppose that \(B_{st}\neq\emptyset\). Define the indicator \(I_{st}\) by \(I_{st}=1\) if \(st\in E(G)\) and \(I_{st}=0\) otherwise. By definition, we have
$|A_s| = |A_t| = d_G(s) - |B_{st}| - I_{st}$
and $|D_{st}| = |B_{st}| + 2n - 2d_G(s) + 2I_{st} - 2$.
Combined with the degree bound \(d_G(s) \ge n - n^{4/5}\), rearranging the second identity yields
\begin{equation}\label{eq:B-D-close}
|B_{st}| \ge |D_{st}| - 2n^{4/5}.
\end{equation}
By Lemma~\ref{lem}(b) and (c), we have
\begin{equation}\label{eq:small-parts}
\sum_{v\in Y_{st}\cup B_{st}\cup\{s,t\}}d_G(v)=O_\ell(n).
\end{equation}

By Lemma~\ref{lem}(e) and the assumption \(B_{st}\neq\emptyset\), all vertices in \(A_{st}\) have pairwise distinct degrees. Therefore, Lemma~\ref{eq:distinct-degree-sum} gives
\begin{equation}\label{Ast}
    \sum_{v\in A_{st}} d_G(v)\le n|A_{st}|-\frac{|A_{st}|^2}{2}+O(n^{3/2}) \le \frac{n^2}{2} + O(n^{3/2}).
\end{equation}

We first consider the case \(\ell=2\). By Lemma~\ref{lem}(b), \(|B_{st}|=O(\sqrt n)\), so \eqref{eq:B-D-close} gives \(|D_{st}|=O(n^{4/5})\). 
Combining \eqref{Ast} with \(|D_{st}|=O(n^{4/5})\) and inequality \eqref{eq:small-parts}, we derive
\[
2|E(G)|=\sum_{v\in D_{st}}d_G(v)
+\sum_{v\in A_{st}}d_G(v)
+\sum_{v\in Y_{st}\cup B_{st}\cup\{s,t\}}d_G(v)
\le \frac{n^2}{2}+O_\ell(n^{9/5}),
\]
which contradicts the assumption that \(|E(G)|\ge (n^2+n)/2\). 

Now, assume \(\ell\ge3\). Define
$D_1:=\{v\in D_{st}:\ |N_G(v)\cap A_{st}|\ge 3\}$ and 
$D_2:=D_{st}\setminus D_1$.
By \eqref{eq:small-parts}, we obtain that
\[
\sum_{v\in D_2} |N_G(v)\cap\left(Y_{st}\cup B_{st}\cup\{s,t\}\right)| \le \sum_{v\in Y_{st}\cup B_{st}\cup\{s,t\}}d_G(v)= O_\ell(n).
\]
Together with the fact that each vertex in \(D_2\) has at most two neighbors in \(A_{st}\) and at most \(|D_{st}|-1\) neighbors in \(D_{st}\), we derive
\begin{equation}\label{D2d}
\sum_{v\in D_2} d_G(v) \le |D_2| |D_{st}| + O_\ell(n).
\end{equation}

We claim that the vertices in \(D_1\) have pairwise distinct degrees. Let \(u,v\) be two vertices in \(D_1\). Since both \(u\) and \(v\) have at least three neighbors in \(A_{st}\), we may choose two distinct vertices
$x\in N_G(u)\cap A_{st}$ and $y\in N_G(v)\cap A_{st}$.
If \(x,y\in A_s\setminus A_{s,4}\), then Lemma~\ref{lem}(a), applied with \(F=\{u,v,s,y\}\) and \(k=2\ell-4\), gives a path \(P\) of length \(2\ell-4\) from \(x\) to some vertex in \(A_s\), avoiding \(F\). Hence $uxPsyv$
is a \(2\ell\)-path, which implies \(d(u)\ne d(v)\). The case \(x,y\in A_t\setminus A_{t,4}\) is similar. It remains to consider the case where \(x,y\) lie on different sides. Without loss of generality, assume $x\in A_s\setminus A_{s,4}$ and $y\in A_t\setminus A_{t,4}$.
Choose any vertex \(w\in B_{st}\). Applying Lemma~\ref{lem}(a) with \(F=\{u,v,s,t,w,y\}\) and \(k=2\ell-6\), we find a path \(P\) of length \(2\ell-6\) from \(x\) to some vertex in \(A_s\), avoiding \(F\). Therefore, $uxPswtyv$ is a \(2\ell\)-path, yielding \(d(u)\ne d(v)\). This proves the claim.

By Lemma~\ref{eq:distinct-degree-sum} and the above claim, we have
\begin{equation}\label{D1d}
    \sum_{v\in D_1}d_G(v) \le n|D_1|-\frac{|D_1|^2}{2}+O(n^{3/2}). 
\end{equation}
Summing \eqref{D2d} and the above inequality, we have
\begin{equation}\label{Dstd}
    \sum_{v\in D_{st}}d_G(v) \le n|D_1|-\frac{|D_1|^2}{2} +|D_2||D_{st}|+O_\ell(n^{3/2}).
\end{equation}

We claim that $|D_{st}|>n/2$. Assume otherwise that $|D_{st}|\le n/2$.
By $|D_{st}|=|D_{1}|+|D_{2}|$, we have
\begin{equation*}
    n|D_1|-\frac{|D_1|^2}{2}+|D_2||D_{st}|-\left(n|D_{st}|-\frac{|D_{st}|^2}{2}\right) = |D_{2}|\left(\frac{|D_1|+|D_{st}|}{2}+ |D_{st}|-n \right)\le 0,
\end{equation*}
where the last inequality follows from $\frac{|D_1|+|D_{st}|}{2}+|D_{st}|-n \le 2|D_{st}|-n \le 0$.
Hence, inequality \eqref{Dstd} implies
\begin{equation*}
    \sum_{v\in D_{st}}d_G(v) \le n|D_{st}|-\frac{|D_{st}|^2}{2} +O_\ell(n^{3/2})\le \frac{3n^2}{8}+O_\ell(n^{3/2}),
\end{equation*}
where the inequality is because $|D_{st}|\le n/2$.
Combining with the bound \eqref{eq:small-parts} and \eqref{Ast}, we derive
\[
2|E(G)| \le \frac{n^2}{2} + \frac{3n^2}{8} + O_\ell(n^{3/2}) < n^2 + n,
\]
which contradicts \(|E(G)| \ge (n^2+n)/2\). Therefore, \(|D_{st}| > n/2\).

Note that \eqref{eq:B-D-close} implies $|A_{st}| \le  2n - |D_{st}| - |B_{st}| \le 2n - 2|D_{st}| + O(n^{4/5})$.
Thus, we have $-2n\le |A_{st}|-2n+2|D_{st}|\le O(n^{4/5})$ and $2n\ge |D_{st}|-\frac{|A_{st}|}{2}\ge 2|D_{st}|-n-O(n^{4/5})\ge -O(n^{4/5})$.
It follows that
$$n|A_{st}|-\frac{|A_{st}|^2}{2}-\left(2n|D_{st}| -2|D_{st}|^2\right)= (|A_{st}|-2n+2|D_{st}|)\left(|D_{st}|-\frac{|A_{st}|}{2}\right)\le O(n^{9/5}).$$
Then inequality \eqref{Ast} gives
\begin{equation}\label{eq:Ast-bound}
\sum_{v\in A_{st}} d_G(v) \le 2n|D_{st}| -2|D_{st}|^2 + O_\ell(n^{9/5}).
\end{equation}
Summing \eqref{eq:small-parts}, \eqref{Dstd}, and \eqref{eq:Ast-bound}, we obtain
\begin{align*}
n^2\le2|E(G)|\le n|D_1|-\frac{|D_1|^2}{2} +|D_2||D_{st}| + 2n|D_{st}|-2|D_{st}|^2  + O_\ell(n^{9/5}).
\end{align*}
Since \(|D_2|=|D_{st}|-|D_1|\), the above inequality is equivalent to
\[ 
\left(n-|D_{st}|-\frac{|D_1|}{2}\right)^2+\frac{|D_1|^2}{4} \le O_\ell(n^{9/5}). 
\] 
Thus, $|n-|D_{st}||=O_\ell(n^{9/10})$ and $|D_1|=O_\ell(n^{9/10})$. 
Then \eqref{Ast} and \eqref{D1d} yield $\sum_{v\in D_1}d_G(v)= O_\ell(n^{19/10})$ and $\sum_{v\in A_{st}}d_G(v)= O_\ell(n^{19/10})$.
Together with \eqref{eq:small-parts}, we deduce that
\begin{equation*}
\sum_{v\in D_2}d_G(v)=2|E(G)|-\sum_{v\in D_1}d_G(v)-\sum_{v\in A_{st}}d_G(v)-\sum_{v\in Y_{st}\cup B_{st}\cup\{s,t\}}d_G(v) \ge n^2-O_\ell(n^{19/10}). 
\end{equation*} 
By \eqref{eq:small-parts} and the fact that each vertex in \(D_2\) has at most two neighbors in \(A_{st}\), we get 
\begin{equation}\label{D2ND}
\sum_{v\in D_2} |N_G(v)\cap D_{st}|\ge \left(\sum_{v\in D_2}d_G(v)\right) - \left(\sum_{v\in Y_{st}\cup B_{st}\cup\{s,t\}}d_G(v)\right) -2|D_2|\ge n^2-O_\ell(n^{19/10}). 
\end{equation} 

Note that $|n-|D_{st}||=O_\ell(n^{9/10})$ and $|D_1|=O_\ell(n^{9/10})$ imply $|n-|D_{2}||=O_\ell(n^{9/10})$.
Hence, \eqref{D2ND} yields
\[
\begin{aligned}
\sum_{v\in D_2}\bigl(|D_{2}|-|N_G(v)\cap D_{2}|\bigr)\le |D_2|^2-\sum_{v\in D_2}\left(|N_G(v)\cap D_{st}|-|D_1|\right) \le O_\ell(n^{19/10}).
\end{aligned}
\]
Consequently, the number of vertices $v$ in $D_2$ satisfying
$|N_G(v)\cap D_{2}|<|D_{2}|-n^{19/20}$
is at most $O_\ell(n^{19/20}).$ That is, there exists $D_2'\subseteq D_2$ with \(|D_2'|=O_\ell (n^{19/20}) \) such that any vertex $v$ in $D_2\backslash D_2'$ satisfies
$|N_G(v)\cap D_{2}|\ge |D_{2}|-n^{19/20}$.
We claim that all vertices in $D_2\backslash D_2'$ have pairwise distinct degrees. Indeed, let $u,v$ be two vertices in $D_2\backslash D_2'$. By the definition of $D_2'$, $G[D_2\backslash D_2']$ has minimum degree at least $|D_{2}|-n^{19/20}-|D_{2}'|\ge n-O_\ell (n^{19/20})$. Thus, $G[D_2\backslash D_2']$ contains a path $uz_1\cdots z_{2\ell-2}$ of length $2\ell-2$ starting at $u$ and avoiding \(v\).
Since both $d_{G[D_2\backslash D_2']}(z_{2\ell-2})$ and $d_{G[D_2\backslash D_2']}(v)$ are at least $n-O_\ell (n^{19/20})\ge \frac{|D_2|}{2}+2\ell$, $z_{2\ell-2}$ and $v$ have at least $4\ell$ common neighbors. Let $w$ be a common neighbor of $z_{2\ell-2},v$ with $w\notin \backslash \{u,z_1,\cdots, z_{2\ell-3}\}$. Then $uz_1\cdots z_{2\ell-2}wv$ is a path of length $2\ell$, giving $d(u)\ne d(v)$, as claimed.

Set $D_2'':=\{v\in D_2:|N_G(v)\backslash D_{st}|\ge n^{19/20}\}$.
Then, by \eqref{eq:small-parts} and the definition of $D_2$, we have
$$|D_2''|n^{19/20}\le \sum_{v\in D_2''} |N_G(v)\backslash D_{st}| \le 2|D_2''|+ \sum_{v\in Y_{st}\cup B_{st}\cup\{s,t\}}d_G(v)= O_\ell (n),$$
which gives $|D_2''|= O_\ell (n^{1/20})$. By the definitions of $D_2'$ and $D_2''$, each vertex in $D_2\backslash (D_2'\cup D_2'')$ has degree at least $|D_{2}|-n^{19/20}\ge n-O_\ell (n^{19/20})$ and at most $|D_{st}|+n^{19/20} \le n+O_\ell (n^{19/20})$.
However, since \(|D_2'|=O_\ell (n^{19/20}) \) and $|D_2''|= O_\ell (n^{1/20})$, we have $|D_2\backslash (D_2'\cup D_2'')|\ge n-O_\ell (n^{19/20})$. Thus, $D_2\backslash (D_2'\cup D_2'')$ must contain two vertices of equal degree, which contradicts the previous claim that all vertices in $D_2\backslash D_2'$ have pairwise distinct degrees.
\end{proof}

\subsection{Estimate of \texorpdfstring{$\beta$}{beta}}

Let $\beta$ denote the maximum integer for which there exist two vertices of degree $\beta$. 

\begin{lem}\label{lem:beta-large}
We have $\beta\ge n-300\ell\sqrt n.$
\end{lem}

\begin{proof}
We first choose a set $S$ greedily.  As long as there are two distinct vertices
$x,y\in V(G)\setminus S$ such that $d_G(x)=d_G(y)$ and $|N_G(x)\cap N_G(y)|\ge 10\ell\sqrt n$, we put $x,y$ into $S$. We claim that fewer than \kern-1mm $\sqrt n$ pairs are chosen.  Otherwise, we can choose $q :=\lceil\sqrt n\rceil$ such pairs, say $\{ x_1,y_1\},\ldots,\{x_q,y_q\}.$ Construct a bipartite graph $H$ whose first part is $\{p_1,\ldots,p_q\}$, where $p_i$ represents the pair $\{x_i,y_i\}$, and whose second part is $V(G)\setminus \{ x_1,y_1, \ldots ,x_q,y_q\}.$ Join $p_i$ to $z$ if $z\in N_G(x_i)\cap N_G(y_i).$ Then $d_H(p_i)\ge 10\ell\sqrt n-2q$ for every $i$, and hence
\[
        e(H)\ge q(10\ell\sqrt n-2q)>(10\ell-3)n.
\]
On the other hand, Theorem~\ref{bipartite-C2k} gives $\operatorname{ex}(q,2n-2q,C_{2\ell})<5\ell n$ for sufficiently large $n$.  Thus, \(H\) contains a cycle of length \(2\ell\), say
$p_1 z_1 p_2 z_2 \cdots p_{\ell} z_\ell p_1$.
By the construction of $H$, $x_1 z_1 x_2 z_2 \cdots x_{\ell} z_\ell y_1$ is a path of length \(2\ell\) in \(G\) with equal-degree endpoints, which yields a contradiction. Therefore,
\begin{equation}\label{eq:S-small}
|S| < 2\sqrt{n}.
\end{equation}
By the maximality of the greedy process, for any $x,y\in V(G)\setminus S$ with $d_G(x)=d_G(y)$, we have
\begin{equation}\label{eq:common-small}
        |N_G(x)\cap N_G(y)|<10\ell\sqrt n .
\end{equation}

By Lemmas \ref{lem:R-minus-small}, the degree of at least $|V(G)\setminus S|-3\sqrt{n}$ vertices in $V(G)\setminus S$ belongs to $\{0,1,\dots,n\}$. By \eqref{eq:S-small}, the set $V(G)\setminus S$ contains two vertices of equal degree. Choose a pair $\{s,t\}\subseteq V(G)\setminus S$ with $d_G(s)=d_G(t)$ such that $d_G(s)$ is as large as possible. For simplicity, we set $A_s:=A_s^{(st)}$ and $A_t:=A_t^{(st)}$. By \eqref{eq:common-small},
\begin{equation}\label{eq:B-small}
        |B_{st}|<10\ell\sqrt n .
\end{equation}

Define $T:=\{x\in (A_s\cup A_t\cup D_{st})\setminus Y_{st}: e(\{ x\},Y_{st})>20\ell\sqrt n\}$.
Then by Lemma \ref{lem} (c),
\begin{equation}\label{eq:T-small}        
|T|<3\sqrt{n}.
\end{equation}
Put $Q:=(A_s\cup A_t\cup D_{st})\setminus (Y_{st}\cup T\cup S)$.
We next prove the key multiplicity estimate. 

\begin{claim}\label{eq:key-threshold}
    For each integer $3\le k\le 6$, there do not exist $k$ vertices in $Q$ with the same degree such that the degree is at least $\frac{|D_{st}|+|B_{st}|}{k-2} + 40\ell\sqrt{n}$.
\end{claim}
\begin{proof}
Suppose, for a contradiction, that such $k$ vertices exist. We first observe that among these $k$ vertices, at most one has at least three neighbors in $A_s\setminus A_{s,4}$. Otherwise, choose distinct $z,z'$ among these $k$ vertices with $e(\{z\},A_s\setminus A_{s,4})\ge 3$ and $e(\{z'\},A_s\setminus A_{s,4})\ge 3$. We may pick $x\in (N_G(z)\cap (A_s\setminus A_{s,4}))\setminus\{z'\}$ and $y\in (N_G(z')\cap (A_s\setminus A_{s,4}))\setminus\{z,x\}$. Applying Lemma~\ref{lem}(a) with $F=\{z,z',s,y\}$ yields a path $P$ of length $2\ell-4$ from $x$ to some vertex in $A_s$ that avoids $F$. Then $zxPsyz'$ is a path of length $2\ell$ with equal-degree endpoints, a contradiction. By symmetry, among these $k$ vertices, at most one has at least three neighbors in $A_t\setminus A_{t,4}$.

Consequently, there exists a set $Z$ of $k-2$ vertices among these $k$ vertices such that $e\bigl(\{z\},(A_s\cup A_t)\setminus Y_{st}\bigr)\le 4$ for all $z\in Z$. Since $z\notin T$, we also have $e(\{z\},Y_{st})\le 20\ell\sqrt{n}$. Together with the contradiction hypothesis, every $z\in Z$ satisfies
\begin{equation}\label{azDst}
    e(\{z\},D_{st}\cup B_{st}\cup\{s,t\}) \ge d_G(z) - 20\ell\sqrt{n} - 4 \ge \frac{|D_{st}|+|B_{st}|}{k-2} + 20\ell\sqrt{n} - 4.
\end{equation}

If $k=3$, then $|Z|=1$, and the lower bound above is greater than $|D_{st}\cup B_{st}\cup\{s,t\}|$, which is impossible. Thus $k\ge 4$, and so $2\le |Z|\le 4$. 
Since $S\cap Q=\emptyset$, \eqref{eq:common-small} shows that any two vertices in $Z$ have fewer than $10\ell \sqrt{n}$ common neighbors.
Thus, by the inclusion–exclusion principle and \eqref{azDst}, we derive
\begin{align*}
    |D_{st}\cup B_{st}\cup\{s,t\}|&\ge \sum_{z\in Z} e(\{z\},D_{st}\cup B_{st}\cup\{s,t\}) - \sum_{\{z,z'\}\subseteq Z,\ z\ne z'} \bigl|N_G(z)\cap N_G(z')\cap (D_{st}\cup B_{st}\cup\{s,t\})\bigr|\\
    &\ge |Z|\left(\frac{|D_{st}|+|B_{st}|}{k-2} + 20\ell\sqrt{n} - 4\right) - \binom{|Z|}{2}\times 10\ell \sqrt{n}\\
    &= |D_{st}|+|B_{st}|+|Z|(20\ell\sqrt{n} - 4-5\ell\sqrt{n}(|Z|-1)).
\end{align*}
As $2\le |Z|\le 4$, the above expression is larger than $|D_{st}|+|B_{st}|+2$, which is impossible. 
\end{proof}

Now let $P:=Q\setminus \{x\in Q:d_G(x)>n\}$.
Then Lemma~\ref{lem:R-minus-small} gives $|Q\setminus P|<3\sqrt{n}$. By \eqref{eq:S-small}, \eqref{eq:T-small}, and Lemma~\ref{lem}(b) and (c), we have
\begin{equation}\label{eq:bad-sum}     
\sum_{x\in V(G)\setminus P}d_G(x)        
\le 2n(|S|+|T|+|Q\setminus P|)+\left(\sum_{x\in Y_{st}} d_G(x)\right)+\left(\sum_{x\in B_{st}}d_G(x)\right)+d_G(s)+d_G(t)
\le 90\ell n^{3/2}.
\end{equation}

We next show that $d_G(s)=d_G(t)>\frac{2n}{5}$.
Suppose otherwise. By the choice of $s$ and $t$, for any integer $d>\frac{2n}{5}$, there is at most one vertex in $V(G)\setminus S$ of degree $d$. Hence, each such degree appears at most once in $P$. Since all vertices in $P$ have degree at most $n$, we obtain
\[
\sum_{x\in P}d_G(x)\le \sum_{i=\left\lfloor\frac{2n}{5}\right\rfloor+1}^{n}i+\left(n+\left\lfloor\frac{2n}{5}\right\rfloor\right)\left\lfloor\frac{2n}{5}\right\rfloor <0.99n^2.
\]
Together with \eqref{eq:bad-sum}, this gives $2|E(G)|<n^2+n$, a contradiction. Hence, $d_G(s)=d_G(t)>\frac{2n}{5}$ holds.

By \eqref{eq:B-small} and the fact that $|D_{st}|=|B_{st}|+2n-2d_G(s)+2I_{st}-2$, we find
\begin{equation}\label{eq:L-upper}
|D_{st}|+|B_{st}|\le 2(n-d_G(s))+20\ell\sqrt{n}.
\end{equation}

Next, we prove $d_G(s)\ge n-300\ell\sqrt{n}$. Assume otherwise.
By the choice of $s$ and $t$, for any integer $d>d_G(s)$, there is at most one vertex in $P$ of degree $d$, as $P\subseteq V(G)\setminus S$. Combining Claim \ref{eq:key-threshold} with inequality \eqref{eq:L-upper}, we deduce that for each $j \in \{1, 2, 3, 4\}$ and each integer $d$ satisfying
\[
d \ge \frac{2(n-d_G(s))}{j} +\frac{20\ell\sqrt{n}}{j} +40\ell\sqrt{n},
\]
there are at most $j+1$ vertices of degree $d$ in $P$.
Let $P'$ be the multiset consisting of the integers $d(w)-\lceil 60\ell\sqrt{n}\rceil$ for each $w\in P$, together with $2n-|P|$ copies of $0$. Then $P'$ has exactly $2n$ elements. Moreover, for each $j \in \{1, 2, 3, 4\}$ and each integer $d\ge 2(n-d_G(s))/j$, there are at most $j+1$ elements equal to $d$ in $P'$. Recall that $\frac{2n}{5}<d_G(s)<n-300\ell\sqrt{n}$. Therefore, applying Lemma \ref{lem3.2} with $S=P'$ and $\alpha =d_G(s)$, we deduce that the sum of all elements in $P'$ is at most $n^2 - 299\ell n^{3/2}$.
By the definition of $P'$, we obtain
\[
\sum_{w\in P} d_G(w) = \lceil 60\ell \sqrt{n}\rceil \cdot |P| + \sum_{d\in P'} d \le 121\ell n^{3/2} + n^2 - 299\ell n^{3/2} = n^2 - 178\ell n^{3/2}.
\]
It follows from \eqref{eq:bad-sum} that,
\[
2|E(G)| = \sum_{w \in V(G)\setminus P} d(w) + \sum_{w \in P} d(w)
\le 90\ell n^{3/2} + n^2 - 178\ell n^{3/2} < n^2 + n,
\]
which contradicts the assumption that $|E(G)| \ge (n^2 + n)/2$.
Thus, $d_G(s)\ge n-300\ell\sqrt{n}$.
Since $d_G(s)= d_G(t)$, the definition of $\beta$ gives $\beta\ge d_G(s)\ge n-300\ell\sqrt n.$ This completes the proof.
\end{proof}

Let $u$ and $v$ be two vertices in $G$ with $d(u)=d(v)=\beta$. We fix this pair $(u,v)$, and define $I_{uv}=1$ if $uv\in E(G)$ and $I_{uv}=0$ otherwise. For brevity, we write $A_u := A_u^{(uv)}$, $A_v := A_v^{(uv)}$, $B := B_{uv}$, $D := D_{uv}$, $Y := Y_{uv}$, $A := A_{uv}$, $A_{u,i}:=A_{u,i}^{(uv)}$, and $A_{v,i}:=A_{v,i}^{(uv)}$, for each $i\in \{1,2,3,4 \}$.

\begin{lem}\label{lem:uv-structure}
The following statements hold.
\begin{enumerate}[label=(\alph*)]
\item We have $B=\emptyset$, $\beta\le n$, and $|D|=2n-2\beta+2I_{uv}-2.$
\item No vertex has a neighbor in \(A_u\) and a neighbor in \(A_v\).
\item $\Delta(G)\le 2n-\beta$, where $\Delta(G)$ denotes the maximum degree of $G$.
\item \(u\) and \(v\) are the only two vertices of degree \(\beta\) in $G$.
\end{enumerate}
\end{lem}

\begin{proof}
(a) By Lemma~\ref{lem:beta-large}, $\beta\ge n-n^{4/5}$ for sufficiently large $n$. Hence, Lemma~\ref{B=kong} yields $B=\emptyset$. 
So, $|D|=2n-2\beta+2I_{uv}-2.$
Moreover, since $2\beta= |A_u|+|A_v|+2I_{uv}\le |V(G)|$, we have $\beta\le n$.

(b) Suppose, for a contradiction, that there exists a vertex $w$ with neighbors $w_1\in N_G(w)\cap A_u$ and $w_2\in N_G(w)\cap A_v$. If $\ell=2$, then $u w_1 w w_2 v$ is a path of length 4 with endpoints of degree $\beta$, a contradiction. Hence $\ell\ge 3$. 

We claim that $d_G(x)\ne d_G(y)$ for all $x\in (A_u\setminus A_{u,4})\setminus\{w,w_1\}$ and $y\in (A_v\setminus A_{v,4})\setminus\{w,w_2\}$. Applying Lemma~\ref{lem}(a) with $F=\{ u,v,w,w_1,w_2,y \}$ and $k=2\ell-6$, we obtain a path $xx_1\cdots x_{2\ell-6}$ avoiding $F$ such that $x_{2\ell-6}\in A_u$. Then
$xx_1\cdots x_{2\ell-6}u w_1 w w_2 v y$
is a $2\ell$-path, giving $d_G(x)\ne d_G(y)$ as claimed.

This, combined with Lemma~\ref{lem}(d), shows that all vertices in $A\setminus\{w,w_1,w_2\}$ have pairwise distinct degrees. Applying Lemma~\ref{eq:distinct-degree-sum} to this set and accounting for the at most three excluded vertices, we get
\[
\sum_{z\in A}d_G(z)
\le \frac{n^2}{2}+O(n^{3/2}).
\]
Moreover, by (a) and Lemma~\ref{lem:beta-large}, $|D| = 2n-2\beta+2I_{uv}-2 \le 600\ell\sqrt{n}$. It follows by Lemma~\ref{lem}(c) that
\begin{align*}
    2|E(G)|&=d_G(u)+d_G(v)+\sum_{z\in A}d_G(z)+\sum_{z\in Y}d_G(z) +\sum_{z\in D}d_G(z) \\
    &\le 4n+\frac{n^2}{2} + O(n^{3/2})+60\ell n+ 2n\times 600\ell\sqrt{n} < n^2+n,
\end{align*}
which is a contradiction.

(c) Let $x$ be a vertex of maximum degree in $G$. 
By (b), without loss of generality, we may assume $N_G(x)\cap A_u=\emptyset$. If $x\notin A_u$, then $x$ is not adjacent to every vertex in $A_u\cup \{x \}$. If $x\in A_u$, then $x$ is not adjacent to all vertices in $A_u\cup\{v\}$. In either case, $x$ has at least $|A_u|+1\ge \beta$ non-neighbors, and hence $d_G(x) \le 2n-\beta$. 

(d) Suppose, to the contrary, that there exists $w\in V(G)\setminus\{u,v\}$ with $d_G(w)=\beta$. 
By Lemma~\ref{lem:beta-large}, we have $\beta\ge n-n^{4/5}> 2n/3$ for sufficiently large $n$. Applying Lemma~\ref{B=kong} to pairs $(u,v)$, $(u,w)$ and $(v,w)$ yields
$B_{uv} = B_{uw} = B_{vw} = \emptyset$.
So, $2n \ge \bigl|N_G(u)\cup N_G(v)\cup N_G(w)\bigr| = d_G(u)+d_G(v)+d_G(w) = 3\beta$, contradicting $\beta>2n/3$.
\end{proof}

Let  $m_d:=\bigl|\{x\in V(G):d_G(x)=d\}\bigr|$ for each integer $d$.

\begin{lem}\label{cl:excess3}
We have
\[
\sum_{\substack{d\ge 4n-4\beta +12 \\ m_d\ge 3}} (m_d-2)d
\le 120\ell n.
\]
\end{lem}

\begin{proof}
Fix $d\ge 4n-4\beta +12$ with $m_d\ge3$. Then Lemma~\ref{lem:uv-structure}(d) gives $d\ne\beta$. 

\begin{claim}
    There do not exist three distinct vertices $z_1,z_2,z_3$ of degree $d$ such that $z_i$ has at most $d/2$ neighbors in $Y$ for each $i=1,2,3$.
\end{claim}
\begin{proof}
Suppose otherwise that such $z_1,z_2,z_3$ exist. Since $d\ne\beta$, none of $z_1,z_2,z_3$ equals $u$ or $v$.
By Lemma~\ref{lem:uv-structure}(a) and the assumption $d\ge 4n-4\beta +12$, for each $i\in \{1,2,3\}$,
\[
e(\{z_i\},A)\ge d - e(\{z_i\},Y) - |D| - 2 \ge \frac{d}{2} - (2n-2\beta+2I_{uv}-2) - 2 \ge 4.
\]
By Lemma~\ref{lem:uv-structure}(b), either $e(\{z_i\}, A_u\setminus A_{u,4})\ge 4$ or $e(\{z_i\}, A_v\setminus A_{v,4})\ge 4$ holds. Without loss of generality, by the Pigeonhole Principle, we may assume that $z_1$ and $z_2$ have at least four neighbors in $A_u\setminus A_{u,4}$. Then, we can choose distinct vertices $x\in N_G(z_1)\cap(A_u\setminus (A_{u,4}\cup \{z_2\}))$ and $y\in N_G(z_2)\cap(A_u\setminus \{z_1\})$. Applying Lemma~\ref{lem} (a) with $F=\{z_1,z_2,y,u\}$ and $k=2\ell-4$, we obtain a path $xx_1\cdots x_{2\ell-4}$ avoiding $F$ with $x_{2\ell-4}\in A_u$. Then $z_1 xx_1\cdots x_{2\ell-4} u y z_2$ is a $2\ell$-path, contradicting $d(z_1)=d(z_2)$.
\end{proof}

By the above claim, there are at least $m_d-2$ vertices of degree $d$ with more than $d/2$ neighbors in $Y$.
It follows that
\[
(m_d-2)\times \frac{d}{2} \le \sum_{d_G(x)=d} e(\{x\},Y).
\]
Summing over all $d\ge 4n-4\beta +12$ with $m_d\ge3$ and using Lemma~\ref{lem}(c), we conclude that
\[
\sum_{\substack{d\ge4(n-\beta)+12\\m_d\ge3}} (m_d-2)d
\le 2\sum_{x\in V(G)} e(\{x\},Y)
= 2\sum_{y\in Y} d_G(y)
\le 120\ell n.
\]
This completes the proof.
\end{proof}

\begin{lem}\label{lem:n-beta-constant}
There exists a constant $c_1=c_1(\ell)$ such that $\beta\ge n-c_1$.
\end{lem}

\begin{proof}
Set $c:=17(300\ell)^2+400\ell+9$. 
If $n-\beta\le c$, the proof is done. Assume that $n-\beta>c$.

Let $V^*$ be the set of vertices of degree at least $2n-\beta-c$.
We claim that $|V^*|\ge 120\ell+2$. Suppose otherwise. 
For every integer $d\ge 4n-4\beta +12$, we keep at most two vertices of degree $d$ and remove all others. 
Denote by $V'$ the set of the remaining vertices. 
By the definition of $\beta$, for each $d\in [\beta+1, 2n-\beta-c-1]$, there is at most one vertex of degree $d$ in $V'$. By the definition of $V'$, there are at most two vertices in $V'$ of degree $d$ for each $d\in[4n-4\beta+12, \beta]$. 
Let $S$ be a multiset of integers consisting of one $d$ for each $d\in [\beta+1, 2n-\beta-c-1]$, two $d$ for each $d\in [4n-4\beta+12, \beta]$, and $8n-8\beta+c+23$ number of $4n-4\beta+11$.
Then, $|S|=2n$. Moreover, the sum of the degrees of all vertices in $V'\backslash V^*$ is at most the sum of all integers in $S$.
Thus,
\begin{equation}\label{middle}
2|E(G)|\le \sum_{w\in V(G)\backslash V'}d_G(w)+\sum_{d\in S}d +\sum_{w\in V^*}d_G(w).
\end{equation}
By Lemma~\ref{cl:excess3} and the definition of $V'$, we find $\sum_{w\in V(G)\backslash V'}d_G(w)\le 120\ell n$.
By Lemma~\ref{lem:uv-structure}(c) and the assumption $|V^*|\le 120\ell +1$, we find $\sum_{w\in V^*}d_G(w)\le (120\ell +1)(2n-\beta)$.
It follows from \eqref{middle} and the definition of $S$ that
$$2|E(G)| \le n^2 + (360\ell +2-c)n + 17(n-\beta)^2 + O_\ell (n-\beta )\le n^2 + O_\ell (\sqrt{n}),$$
where the last inequality follows from Lemma~\ref{lem:beta-large} and the definition of $c$. This contradicts $|E(G)|\ge (n^2+n)/2$ and completes the claim.

Now, suppose $n-\beta>c+\ell+1$ (Otherwise, the proof is done). Then $V^*\subseteq R$, as $2n-\beta-c > n+\ell+1$.  
Hence, $r\ge 120\ell+2\ge\ell+1$.
We claim that at most $120\ell$ vertices in $V^*$ have no neighbors in $A$. Recall that $A:=(A_u\cup A_v)\backslash Y$. For any $w\in V^*$ with no neighbors in $A$,
\[
e(\{ w\},Y)\ge d_G(w) - |D| - 2 \ge 2n-\beta-c - (2n-2\beta) - 2 = \beta - c - 2 \ge \frac{n}{2},
\]
where we have used Lemma~\ref{lem:uv-structure}(a) and Lemma~\ref{lem:beta-large}. The claim then follows from Lemma~\ref{lem}(c) and the fact that $\sum_{w\in V(G)}e(\{ w\},Y)=\sum_{v\in Y}d_G(v)$. We may thus choose two distinct vertices $w_1,w_2\in V^*$, each with at least one neighbor in $A$.

Without loss of generality, assume $N_G(w_1)\cap(A_v\setminus A_{v,4})\ne \emptyset$.
Pick $r_0\in N_G(w_1)\cap(A_v\setminus A_{v,4})$. By Lemma~\ref{lem:uv-structure}(b), $N_G(w_1)\cap A_u=\emptyset$, so $N_G(w_1)\subseteq A_v\cup D\cup\{u,v\}$. 
By the fact that $|A_v\cup D\cup\{u,v\}|=2n-\beta +I_{uv}$ and $d_G(w_1)\ge 2n-\beta -c$, we have $|A_v\setminus N_G(w_1)|\le c+1$ and $|D\setminus N_G(w_1)|\le c+1$. The same argument applies to $w_2$ on its respective side, giving $|D\setminus N_G(w_2)|\le c+1$.

Define $Z := \bigl( (N_G(w_1)\cap A_v) \cup (D\cap N_G(w_1)\cap N_G(w_2)) \bigr) \setminus\{r_0\}$. Then
\begin{equation}\label{Z}
|Z|\ge |A_v| + |D| - 3(c+1) - 1 = 2n-\beta+I_{uv} - 3c - 6 \ge 2n-\beta - 3c - 6.
\end{equation}

We claim that vertices in $Z$ have pairwise distinct degrees.
Let $x,y$ be two vertices in $Z$.
If one of $x$ and $y$ lies in $A_v$, say $x\in A_v$, then applying Lemma~\ref{lem}(a) with $F=\{x,y,w_1,v\}$ and $k=2\ell-4$ yields a path $P$ of length $2\ell-4$ from $r_0$ to some $r'\in A_v$ avoiding $F$. Then $y w_1 r_0 P r' v x$ is a $2\ell$-path, giving $d(x)\ne d(y)$. It remains to consider the case $x,y\in D\cap N_G(w_1)\cap N_G(w_2)$. Since $w_1,w_2$ lie in $V^*\subseteq R$ and $r\ge \ell+1$, Lemma~\ref{differentdegree} gives $d(x)\ne d(y)$, as claimed.

Write $R=\{v_1,\ldots,v_r\}$ with $w_1=v_j$ for some $j\in\{ 1,\ldots,r\}$. By Lemma~\ref{lem:U}, we can choose a vertex $u_i\in U_{v_i}$ for each $i\ne j$. 
Since $Z\subseteq N(v_j)$ and $U_{v_i}\subseteq N(v_i)$ for each $i\ne j$, Lemma~\ref{differentdegree} shows that any vertex in $\{u_i:i\ne j\}$ and any vertex in $Z\cup\{u_i:i\ne j\}$ have different degrees. Together with the above claim, we conclude that all vertices in $Z\cup\{u_i:i\ne j\}$ have pairwise distinct positive degrees. 
By the definition of $U_{v_i}$, we have $|Z\cup\{u_i:i\ne j\}|=|Z|+r-1$.
By Lemma~\ref{eq:v1bound}, we derive $|Z|+r-1\le n+\ell +r$.
It follows from \eqref{Z} that $\beta\ge n-3c-\ell-7$, completing the proof.
\end{proof}

\begin{lem}\label{lem:large-good-vertices}
Set $c_2:= 180\ell+2c_1+4$. Then there exists a vertex $w\in A_u\setminus A_{u,4}$ of degree at least $\beta-c_2$ such that $N_G(w)\cap A\ne \emptyset$. The analogous statement holds with \(u\) and \(v\) interchanged.
\end{lem}

\begin{proof}
By Lemma~\ref{lem:n-beta-constant} and Lemma~\ref{lem:uv-structure}(a), $0\le n-\beta\le c_1$. 
By symmetry, it suffices to prove the statement for $A_u\setminus A_{u,4}$. 
Suppose, for a contradiction, that every vertex in $A_u\setminus A_{u,4}$ of degree at least $\beta-c_2$ has no neighbors in $A$.
By Lemma~\ref{lem:uv-structure}(a), $B=\emptyset$ and $|D|\le 2(n-\beta)\le 2c_1$. Thus, for any $w\in A_u\setminus A_{u,4}$ with $d_G(w)\ge\beta-c_2$ we have
\[
e(\{w\},Y)\ge d_G(w) - |D| - 2 \ge \beta - c_2 - 2c_1 - 2 \ge \frac{n}{2},
\]
where the last inequality follows from Lemma~\ref{lem:beta-large}.
Therefore, by Lemma~\ref{lem}(c) and the fact that $\sum_{w\in V(G)}e(\{ w\},Y)=\sum_{v\in Y}d_G(v)$, there are at most $120\ell$ vertices in $A_u\setminus A_{u,4}$ of degree at least $\beta-c_2$.
It follows by Lemma~\ref{lem}(d) and Lemma~\ref{lem:uv-structure}(c) that
\begin{equation}\label{wAuAu}
\sum_{w\in A_u\setminus A_{u,4}}d_G(w)
\le 120\ell(2n-\beta)+\sum_{d=1}^{\beta-c_2-1}d.
\end{equation}
Moreover, Lemma~\ref{lem}(d) and Lemma~\ref{lem:uv-structure}(c) give
\begin{equation}\label{wAvAv}
\sum_{w\in A_v\setminus A_{v,4}}d_G(w)\le \sum_{d=1}^{2n-\beta}d.
\end{equation}

By Lemma~\ref{lem:uv-structure}(a) and Lemma~\ref{lem:n-beta-constant}, we have $|D|\le 2(n-\beta)\le 2c_1$. Then, Lemma~\ref{lem:uv-structure}(c) yields $\sum_{w\in D}d_G(w)\le 2c_1(2n-\beta)$. Together with \eqref{wAuAu}, \eqref{wAvAv}, and Lemma~\ref{lem}(c), we deduce that
\begin{align*}
2|E(G)|&\le 2\beta +120\ell(2n-\beta)+\left(\sum_{d=1}^{\beta-c_2-1}d\right)+\left(\sum_{d=1}^{2n-\beta}d \right)+ 60\ell n + 2c_1(2n-\beta) \\
&= n^2 + (180\ell+2c_1+2-c_2)n+(n-\beta)^2 + (120\ell + 2c_1 + c_2-1)(n-\beta)+ O_\ell(1)< n^2 + n,
\end{align*}
where the last inequality follows from Lemma~\ref{lem:n-beta-constant} and the definition of $c_2$.
This contradicts the assumption that $|E(G)|\ge (n^2+n)/2$.
\end{proof}

\begin{lem}\label{w1w2}
There exist four vertices $w_1,w_2,x,y$ such that
$w_1,w_2\in A$; $d_G(w_1),d_G(w_2)\ge\beta-c_2$; 
$x\in N_G(w_1)\cap(A_u\setminus A_{u,4})$; and $y\in N_G(w_2)\cap(A_v\setminus A_{v,4})$.
\end{lem}

\begin{proof}
    Let $c_1$ be given by Lemma~\ref{lem:n-beta-constant}, and let \(c_2\) be given by Lemma~\ref{lem:large-good-vertices}. By Lemma~\ref{lem:large-good-vertices}, we can choose two vertices $w_1\in A_u\setminus A_{u,4},\ w_2\in A_v\setminus A_{v,4}$ each of degree at least $\beta-c_2,$ along with two vertices $ x\in N_G(w_1)\cap A, y\in N_G(w_2)\cap A$. We claim that $\{x,y\}\nsubseteq A_u\setminus A_{u,4}$. Suppose otherwise that $\{x,y\}\subseteq A_u\setminus A_{u,4}$.
By Lemma~\ref{lem:uv-structure}(b),
$N_G(w_1)\cap A_v = N_G(w_2)\cap A_v = \emptyset$.
Hence, for \(i\in\{1,2\}\),
\[
e(\{w_i\},A_u)\ge d_G(w_i)-|D|-2 \ge\beta-c_2-2c_1-2,
\]
where we have used $|D|\le 2n-2\beta\le 2c_1$.
It follows that
\begin{align*}
|N_G(w_1)\cap N_G(w_2)\cap A_u|\ge e(\{w_1\},A_u)+e(\{ w_2\},A_u)-|A_u|\ge\beta-2c_2-4c_1-4+I_{uv}>0.
\end{align*}
This contradicts Lemma~\ref{lem:uv-structure}(b), as $w_1\in A_u$ and $w_2\in A_v$. By symmetry, we have $\{x,y\}\nsubseteq A_v\setminus A_{v,4}$.
Therefore, relabelling if necessary, we obtain the desired vertices.
\end{proof}

\begin{lem}\label{lem:half-graph-structure}
We have $\beta=n $ and $G\cong H_n.$
\end{lem}

\begin{proof}
Let $w_1,w_2,x,y$ be given by Lemma~\ref{w1w2}.
Note that Lemma~\ref{lem:uv-structure}(b) implies $N_G(w_1)\cap A_v=N_G(w_2)\cap A_u=\emptyset$.
Set $C_u:=\bigl(N_G(w_1)\cap A_u\bigr)\setminus\{x\}$ and
$C_v:=\bigl(N_G(w_2)\cap A_v\bigr)\setminus\{y\}$.
Then we have
\[
|C_u|\ge d_G(w_1)-|D|-3\ \ 
{\rm and}\ \ 
|C_v|\ge d_G(w_2)-|D|-3.
\]
Together with the facts that $|A_u|\le\beta$, $|D|\le 2(n-\beta)\le 2c_1$, and $d_G(w_1),d_G(w_2)\ge\beta-c_2$, we deduce that
\begin{equation}\label{Cuv}
|A_u\setminus C_u|,\ |A_v\setminus C_v|\le \beta - (\beta - c_2 - 2c_1 - 3)
= c_2 + 2c_1 + 3.
\end{equation}

Set $c_3:=2c_2+6c_1+12.$

\begin{claim}\label{cl:high-degree}
For any integer $d>c_3$, there are at most two vertices in $G$ with degree $d$.
\end{claim}

\begin{proof}
Define $S_u:=\{z\in V(G):e(\{z\},C_u)\ge3\}$ and $S_v:=\{z\in V(G):e(\{z\},C_v)\ge3\}$. Note that Lemma~\ref{lem:uv-structure}(b) implies $S_u\cap S_v=\emptyset$.

We assert that the vertices in \(S_u\) have pairwise distinct degrees. Let \(p,q\) be two vertices in $S_u$.
Choose two distinct vertices $a\in (N_G(p)\cap C_u)\backslash \{q\}$ and $\ b\in (N_G(q)\cap C_u)\backslash \{ q\}$.
Note that \(v\notin S_u\). If one of \(p,q\) is \(u\), then Lemma~\ref{lem:uv-structure}(d) yields $d(p)\ne d(q)$.
Now, assume that neither $p$ nor $q$ is $u$.
If \(\ell=2\),  since $p a u b q$ is a path of length four, we have $d(p)\ne d(q)$. Now, assume that \(\ell\ge3\). 

First, consider the case $p,q\notin\{w_1,x\}$. 
Applying Lemma~\ref{lem} (a) with $F=\{p,q,a,b,w_1,u\}$ and $k=2\ell-6,$ there is a path $xz_1\cdots z_{2\ell-6}$ with $z_{2\ell-6}\in A_u$, avoiding \(F\). Hence,
$paw_1xz_1\cdots z_{2\ell-6}ubq$
is a \(2\ell\)-path, and $d(p)\ne d(q)$.

Secondly, suppose that \(p=w_1\) and \(q\ne x\). By Lemma~\ref{lem}(a), there is a path $xz_1\cdots z_{2\ell-4}$ with $z_{2\ell-4}\in A_u$ avoiding $\{w_1,q,b,u\}$. Then
$w_1xz_1\cdots z_{2\ell-4}u b q$ is a \(2\ell\)-path, giving $d(p)\ne d(q)$.

Finally, suppose that \(p=x\). Since \(x=p\in S_u\), it has a neighbor in \(C_u\subseteq A_u\). 
Since $xw_1\in E(G)$, $xa\in E(G)$, and $a\in A_u$, Lemma~\ref{lem:uv-structure}(b) implies $w_1\in A_u$. Since $w_1\in A$, we have \(w_1\in A_u\setminus A_{u,4}\). 
If $q\ne w_1$, by Lemma~\ref{lem}(a), there is a $2\ell-4$-path $w_1z_1\cdots z_{2\ell-4}$ with $z_{2\ell-4}\in A_u$, avoiding $\{x, q,b,u\}$. So, $xuz_{2\ell-4}z_{2\ell-5}\cdots w_1bq$ is a \( 2\ell\)-path, and $d(p)\ne d(q)$.
If $q= w_1$, then Lemma~\ref{lem}(d) gives $d(w_1)\ne d(x)$, and the assertion holds. 
By symmetry, the vertices in \(S_v\) have pairwise distinct degrees.

By \eqref{Cuv} and the definitions of $S_u$ and $S_v$, for any $z\in V(G)\backslash (S_u\cup S_v)$, we have
\[
d_G(z)\le 2+2+|A_u\setminus C_u|+|A_v\setminus C_v|+|D|+2\le 2+2+2(c_2+2c_1+3)+2c_1+2=c_3,
\]
where we have used $|D|\le 2n-2\beta\le 2c_1$.
So, any vertex of degree greater than $c_3$ lies in $S_u\cup S_v$.
By the above assertion, every degree greater than $c_3$ occurs at most once in each of $S_u$ and $S_v$.
\end{proof}

We next show that there exists a vertex of degree $2n-\beta$. Suppose otherwise. 
By Lemma~\ref{lem:uv-structure}(c), we have $\Delta(G)\le 2n-\beta-1$. 
Note that $\beta\le\Delta(G)$ implies $\beta\le n-1$.
By the definition of $\beta$, each degree in $[\beta +1, 2n-\beta-1]$ occurs at most once. 
Moreover, by Claim~\ref{cl:high-degree}, each degree in $[c_3+1,\beta]$ occurs at most twice. 
Consequently,
\begin{equation}\label{d2d2c3}
2|E(G)|\le \sum_{d=\beta+1}^{2n-\beta-1} d + 2\left(\sum_{d=c_3+1}^{\beta} d\right) + (2c_3+1)c_3 = n^2 + (n-\beta)^2 - (n-\beta) + c_3^2,
\end{equation}
where we have used $(2n-2\beta-1) + 2(\beta-c_3) + (2c_3+1) = |V(G)|$. By Lemma~\ref{lem:n-beta-constant}, \eqref{d2d2c3} yields $2|E(G)|< n^2+n$, a contradiction. Hence, there exists a vertex, say $w_0$, such that $d(w_0) = 2n-\beta$.

\begin{claim}
    We have $\beta =n$.
\end{claim}
\begin{proof}
Suppose, to the contrary, that $\beta\le n-1$. Then, $w_0\notin \{u,v \}$.
By Lemma~\ref{lem:uv-structure}(b), we may assume, without loss of generality, that $N_G(w_0)\cap A_u=\emptyset$. If \( w_0\in A_v\), then
\[
d(w_0)\le (|A_v|-1)+|D|+1=2n-\beta+I_{uv}-2<2n-\beta,
\]
contradicting $d(w_0) = 2n-\beta$. If \(w_0\in D\), then
\[
d(w_0)\le |A_v|+|D|-1= 2n-\beta+I_{uv}-3 <2n-\beta,
\]
again a contradiction. Hence, we have \(w_0\in A_u\).
Then, $N_G(w_0)\subseteq A_v\cup D\cup\{u\}$. However,
\[
|A_v\cup D\cup\{u\}|=|A_v|+|D|+1
=2n-\beta+I_{uv}-1\le 2n-\beta =d(w_0).
\]
It follows that $I_{uv}=1$ and $N_G(w_0)=A_v\cup D\cup\{u\}$.
Therefore, Lemma~\ref{lem:uv-structure}(b) shows that any vertex in $A_v\cup D$ has no neighbors in $A_v$, since $w_0\in A_u$. That is, $e(A_v,A_v\cup D)=0$.
So, the degree of each vertex in $A_v$ is at most $|A_u|+1=\beta$.
Together with Lemma~\ref{lem:uv-structure}(d) and the fact that every vertex in $A_v$ is adjacent to both $v$ and $w_0$, we conclude that every vertex in $A_v$ has degree at least 2 and at most $\beta-1$.
Since $|A_v|=\beta -1$, by the Pigeonhole Principle, there exist two vertices, say $a$ and $b$, in $A_v$ of the same degree.

Recall that $y\in A_v\setminus A_{v,4}$. If $y\notin\{a,b\}$, then applying Lemma~\ref{lem}(a) with $F=\{a,b,v,w_0\}$ and $k=2\ell-4$ yields a path $yz_1\cdots z_{2\ell-4}$ with $z_{2\ell-4}\in A_v,$ that avoids $F$. Then,
$bw_0yz_1\cdots z_{2\ell-4}va$ is a $2\ell$-path with equal-degree endpoints, a contradiction. Without loss of generality, assume $a=y$. By Lemma~\ref{lem}(a), there is a path $yz_1\cdots z_{2\ell-2}$ with $z_{2\ell-2}\in A_v,$ avoiding $\{b,v\}$. Then $yz_1\cdots z_{2\ell-2}v b$
is a forbidden path of length $2\ell$. This completes the proof.
\end{proof}

Since \(\beta=n\), Lemma~\ref{lem:uv-structure} gives
$B=D=\emptyset$, $uv\in E(G)$, and $\Delta(G)\le n$.
Moreover, \(u\) and \(v\) are the only two vertices of degree \(n\).

We next show that $|A_u\setminus A_{u,4}|\ge 4$ and $|A_v\setminus A_{v,4}|\ge 4$.
By Lemma~\ref{lem}(d) and $\Delta(G)\le n$, we have $\sum_{z\in A_v\backslash A_{v,4}}d_G(z)\le (n^2+n)/2$.
If \(|A_u\setminus A_{u,4}|\le3\), then $\sum_{z\in A_u\backslash A_{u,4}}d_G(z)\le 3n$. It follows from Lemma~\ref{lem}(c) that
\[
2|E(G)|\le \frac{n^2+n}{2}+3n+d_G(u)+d_G(v)+\sum_{z\in Y}d_G(z)
\le \frac{n^2+n}{2}+3n+2n+60\ell n<n^2+n,
\]
for sufficiently large $n$, a contradiction. Thus, $|A_u\setminus A_{u,4}|\ge 4$. By symmetry, we have $|A_v\setminus A_{v,4}|\ge 4$.

\begin{claim}\label{nminus1}
There exist unique vertices $w_u\in A_u$ and $w_v\in A_v$ of degree \(n-1\). Moreover, we have $N_G(w_u)\cap A_u=\emptyset$ and $N_G(w_v)\cap A_v=\emptyset.$
\end{claim}

\begin{proof}
We first prove that at least two vertices in $G$ have degree $n-1$. Suppose otherwise. By Claim~\ref{cl:high-degree}, each degree in \([c_3+1,n-2]\) occurs at most twice, while degree \(n\) occurs exactly twice. Hence,
\[
2|E(G)|\le2n+(n-1)+2\left(\sum_{d=c_3+1}^{n-2}d\right)+(2c_3+1)c_3= n^2+c_3^2+1 < n^2+n,
\]
a contradiction. So, the degree \(n-1\) occurs at least twice. Indeed, by Claim~\ref{cl:high-degree}, there are exactly two vertices, say $w_u$ and $w_v$, of degree $n-1$.
By Lemma~\ref{B=kong}, we have $B_{w_u w_v}=\emptyset$. 
In particular, neither $u$ nor $v$ is a common neighbor of $w_u$ and $w_v$. So, $\{w_u,w_v\}\nsubseteq A_u$ and $\{w_u,w_v\}\nsubseteq A_v$. Without loss of generality, we may assume that $w_u\in A_u$ and $w_v\in A_v$. 


Now, it suffices to prove $N_G(w_u)\cap A_u=\emptyset$. 
Suppose otherwise. Then by Lemma~\ref{lem:uv-structure}(b), we have $N_G(w_u)\cap A_v=\emptyset$, and hence $N_G(w_u)=(A_u\setminus\{w_u\})\cup\{u\}$. 
Recall that $|A_u\setminus A_{u,4}|\ge 4$.
We can choose a vertex $h\in(A_u\setminus A_{u,4})\setminus\{w_u\}$. By Lemma~\ref{lem}(a), there exists a $(2\ell-4)$-path $P$ from $h$ to some $r\in A_u$ that avoids $\{w_u,w_v,u,v\}$. Then $w_u h P r u v w_v$ is a path of length $2\ell$ with equal-degree endpoints, a contradiction. Thus, $N_G(w_u)\cap A_u=\emptyset$. By symmetry, we have $N_G(w_v)\cap A_v=\emptyset$.
\end{proof}

By Claim~\ref{nminus1}, \(w_u\) is adjacent to all but one vertex, say $w_u'$, in \(A_v\). 
For any two vertices in $A_v$, at least one of them is adjacent to $w_u$. Hence, by Lemma~\ref{lem:uv-structure}(b), these two vertices are not adjacent, since $w_u\in A_u$.
It follows that $e(G[A_v])=0$. By symmetry, we have $e(G[A_u])=0$.

We next show that all vertices in $A_v$ have pairwise distinct degrees. Let $a,b$ be two vertices in $A_v$. 
Without loss of generality, assume that $b\ne w_u'$, and hence $w_ub\in E(G)$. Moreover, since $|A_v\setminus A_{v,4}|\ge4$, we may choose a vertex 
$h\in (A_v\setminus A_{v,4})\setminus\{a,b,w_u' \}$.
Applying Lemma~\ref{lem}(a), we find a path $P$ of length $2\ell-4$ from $h$ to some $r\in A_v$, avoiding $\{ a,b,v,w_u\}$. Then $bw_uhPrva$ is a path of length $2\ell$ with equal-degree endpoints, giving $d(a)\ne d(b)$. Thus, the vertices in $A_v$ have pairwise distinct degrees.
Recall that $\Delta(G)\le n$, and \(u,v\) are the only two vertices of degree \(n\).
So, we can label $A_v=\{v_1,\cdots,v_{n-1}\}$ such that \( d(v_j)=j\) for each $1\le j\le n-1$. By symmetry, we can label $A_u=\{u_1,\cdots,u_{n-1}\}$ such that \( d(u_j)=j\) for each $1\le j\le n-1$. 

Note that \(d(v_1)=1\) implies \(N(v_1)=\{v\}\). Since \(d(u_{n-1})=n-1\) and $e(G[A_u])=0$, we deduce that \(N(u_{n-1})=\{u\}\cup (A_v\setminus \{v_1\})\). Then $d(v_2)=2$ implies \(N(v_2)=\{v,u_{n-1}\}\). Since \(d(u_{n-2})=n-2\) and $e(G[A_u])=0$, exactly two vertices in \(A_v\) are not adjacent to $u_{n-2}$. It follows that \(N(u_{n-2})=\{u\}\cup (A_v\setminus \{v_1,v_2\})\). Repeating this argument, we obtain
$N(u_{n-i})=\{u\}\cup \big(A_v\setminus \{v_1,v_2,\dots,v_i\}\big)$,
for each \(i\in \{1,2,\dots,n-1\}\). The resulting graph is exactly the half graph \(H_n\) with bipartition \(\{u_1,\dots,u_{n-1},v\}\) and \(\{v_1,\dots,v_{n-1},u\}\). This completes the proof.
\end{proof}

\section*{Acknowledgements}
The authors wish to thank Xiangshuai Dong from Xiamen University for helpful discussions and valuable comments on an early draft of this paper, particularly on the 4-path case. Research supported by National Key Research and Development Program of China (Grant Nos. 2023YFA1010202) and National Natural Science Foundation of China (Grant Nos. 125B1009, 12371342).

\end{document}